\documentclass[a4paper,11pt]{amsart}
\usepackage{mathrsfs}
\usepackage{appendix}
\usepackage{amssymb}
\usepackage{fancyhdr}
\usepackage{charter}
\usepackage{typearea}
\usepackage{color}
\usepackage{amsmath,amstext,amsthm,amscd}

\usepackage[a4paper,top=3cm,bottom=3cm,left=3cm,right=3cm]{geometry}

\makeatletter
      \def\@setcopyright{}
      \def\serieslogo@{}
      \makeatother

\newcommand{\Complex}{\mathbb C}
\newcommand{\Real}{\mathbb R}

\newcommand{\ddbar}{\overline\partial}
\newcommand{\pr}{\partial}
\newcommand{\ol}{\overline}
\newcommand{\Td}{\widetilde}
\newcommand{\norm}[1]{\left\Vert#1\right\Vert}
\newcommand{\abs}[1]{\left\vert#1\right\vert}

\newcommand{\set}[1]{\left\{#1\right\}}
\newcommand{\To}{\rightarrow}
\DeclareMathOperator{\Ker}{Ker}

\theoremstyle{plain}
\newtheorem{theorem}{Theorem}[section]
\newtheorem{lemma}[theorem]{Lemma}
\newtheorem{corollary}[theorem]{Corollary}
\newtheorem{proposition}[theorem]{Proposition}

\newtheorem{definition}[theorem]{Definition}

\newtheorem{remark}[theorem]{Remark}
\newtheorem{ass}[theorem]{Assumption}
\numberwithin{equation}{section}

\begin{document}

\title[$S^1$-equivariant Index theorems and Morse inequalities on complex manifolds with boundary]
{$S^1$-equivariant Index theorems and Morse inequalities on complex manifolds with boundary}

\author[]{Chin-Yu Hsiao}
\address{Institute of Mathematics, Academia Sinica and National Center for Theoretical Sciences, 6F, Astronomy-Mathematics Building,
No.1, Sec.4, Roosevelt Road, Taipei 10617, Taiwan}
\email{chsiao@math.sinica.edu.tw or chinyu.hsiao@gmail.com}
\thanks{Chin-Yu Hsiao was partially supported by Taiwan Ministry of Science and Technology project 106-2115-M-001-012 and Academia Sinica Career Development
Award. }
\author[]{Rung-Tzung Huang}
\thanks{Rung-Tzung Huang was supported by Taiwan Ministry of Science and Technology projects 105-2115-M-008-008-MY2 and 107-2115-M-008-007-MY2.}
\address{Department of Mathematics, National Central University, Chung-Li 320, Taiwan}

\email{rthuang@math.ncu.edu.tw}

\author[]{Xiaoshan Li}
\address{School of Mathematics
and Statistics, Wuhan University, Wuhan 430072, Hubei, China}
\thanks{Xiaoshan Li was supported by  National Natural Science Foundation of China (Grant No. 11871380,  No. 11501422)}
\email{xiaoshanli@whu.edu.cn}
\author[]{Guokuan Shao}

\address{School of Mathematics (Zhuhai), Sun Yat-sen University, Zhuhai 519082, Guangdong, China}

\email{shaogk@mail.sysu.edu.cn}

\begin{abstract}
Let $M$ be a complex manifold of dimension $n$ with smooth connected boundary $X$. Assume that $\ol M$ admits a holomorphic $S^1$-action preserving the boundary $X$ and the $S^1$-action is transversal on $X$. We show that the $\ddbar$-Neumann Laplacian on $M$ is transversally elliptic and as a consequence, the $m$-th Fourier component of the $q$-th Dolbeault cohomology group $H^q_m(\ol M)$ is finite dimensional, for every $m\in\mathbb Z$ and every $q=0,1,\ldots,n$. This enables us to define  $\sum^{n}_{j=0}(-1)^j{\rm dim\,}H^j_m(\ol M)$ the $m$-th Fourier component of the Euler characteristic on $M$ and to study large $m$-behavior of $H^q_m(\ol M)$.  In this paper, we establish an index  formula for $\sum^{n}_{j=0}(-1)^j{\rm dim\,}H^j_m(\ol M)$ and Morse inequalities for $H^q_m(\ol M)$.
\end{abstract}

\maketitle \tableofcontents

 \section{Introduction and statement of the main results}\label{s-gue170929}
\subsection{Introduction}
Let $M$ be a complex manifold of dimension $n$ with smooth connected boundary $X$. The study of holomorphic sections on $\ol M$ is an important subject in complex analysis, complex geometry and is closely related to deformation and embedding problems on complex manifolds with boundary (see~~\cite{Be05},~\cite{EH00},~\cite{Ma96},~\cite{Ma16}).
 The difficulty comes from the fact that the  associated $\ddbar$-Neumann Laplacian on $M$ can be non-hypoelliptic and it is very difficult to understand the space of holomorphic sections. A clue to the above phenomenon arises from the following.  By $\overline{\partial }^{2}=0$, one has the
 $\overline{\partial }$-complex: $\cdots \rightarrow \Omega ^{0,q-1}(\ol M)\rightarrow \Omega
^{0,q}(\ol M)\rightarrow \Omega ^{0,q+1}(\ol M)\rightarrow \cdots $ and the $q$-th Dolbeault cohomology group: $H^{q}(\ol M):=\frac{\mathrm{Ker\,}\overline{%
\partial } : \Omega ^{0,q}(\ol M)\rightarrow \Omega ^{0,q+1}(\ol M)}{\mathrm{Im\,}%
\overline{\partial }:\Omega ^{0,q-1}(\ol M)\rightarrow \Omega ^{0,q}(\ol M)}$.
As in complex manifolds without boundary case, to understand the space $H^{0}(\ol M)$, one could try
to establish a Hirzebruch-Riemann-Roch theorem for
$\sum\limits_{j=0}^{n}(-1)^{j}{\rm dim\,}H^{j}(\ol M)$, an analogue of the Euler characteristic, and
to prove vanishing
theorems for $H^{j}(\ol M)$, $j\geq 1$. The first difficulty with
such an approach lies in the
fact that ${\rm dim\,}H^{j}(\ol M)$ could be infinite for some $j$.

Another line of thought lies in the fact that the space $H^{q}(\ol M)$ is related to the $\ddbar$-Neumann Laplacian
\[\Box^{(q)}=\ddbar^\star\,\ddbar+\ddbar\,\ddbar^\star: {\rm Dom\,}\Box^{(q)}\subset L^2_{(0,q)}(M)\To L^2_{(0,q)}(M).\]
One can try to study the distribution kernel of the orthogonal projection onto ${\rm Ker\,}\Box^{(q)}$ (Bergman kernel) and the associated heat operator $e^{-t\Box^{(q)}}$.  Unfortunately without any Levi
curvature assumption, $\Box^{(q)}$ is not hypoelliptic in general and
it is unclear how to study  the kernel of $\Box^{(q)}$ and  $e^{-t\Box^{(q)}}$.

Assume that $\ol M$ admits a holomorphic compact Lie group action $G$.
The key observation in this paper
is the following: If all the "non-hypoelliptic directions" of the $\ddbar$-Neumann Laplacian are contained in the space of vector fields on $X$ induced by the Lie algebra of $G$, then we can show that $\Box^{(q)}$ is transversally elliptic in the sense of  Atiyah and Singer without any Levi curvature assumption. In this case, it is possible to study $G$-equivariant holomorphic sections on $\ol M$, $G$-equivariant Bergman and heat kernels on $\ol M$ and to overcome the difficulty mentioned above.  On the other hand, the study of $G$-equivariant holomorphic sections on $\ol M$ has its own  interest and is closely related to  $G$-equivariant embedding and deformation problems (see~\cite{BSW78},~\cite{Ham79}), geometric quantization theory and $G$-equivariant index theorems on complex manifolds with boundary. In this paper, we restrict ourselves to the $S^1$-action case. It should be mentioned that it is possible to study general compact Lie group action case by using the method developed in this paper.

Suppose that $\ol M$ admits a holomorphic $S^{1}$-action $e^{i\theta}$ preserving  the boundary $X$ and the $S^{1}$-action is transversal on $X$. For every $m\in\mathbb Z$, let $H^q_m(\ol M)$ be the $m$-th Fourier component or representation of the $q$-th Dolbeault cohomology group with respect to the $S^1$-action (see \eqref{e-gue170929IId}).
We will show in Section~\ref{s-gue171006} that ${\rm dim\,}H^{q}_{m}(\ol M)<+\infty$, for every $m\in\mathbb Z$ and every $q=0,1,\ldots,n$.
To find an exact formula for the Euler characteristic $\sum^{n}_{j=0}(-1)^j{\rm dim\,}H^j_m(\ol M)$ and to establish Morse inequalities for $H^q_m(\ol M)$ are very natural and fundamental problems.
In this paper, we successfully  establish exact formula for $\sum^{n}_{j=0}(-1)^j{\rm dim\,}H^j_m(\ol M)$ and Morse inequalities for $H^q_m(\ol M)$. We believe that our results will have some applications in complex geometry and geometric quantization theory.  It should be noticed that the index formula for $\sum^{n}_{j=0}(-1)^j{\rm dim\,}H^j_m(\ol M)$ has its own interest. To the authors' knowledge, the index formula proved in the present paper is the first general equivariant index theorem for a transversally elliptic operator on a complex manifold with boundary.  We describe our approach briefly. In Section~\ref{s-gue171010}, we introduce a new operator $\hat\ddbar_{\beta}: \Omega^{0, q}(X)\rightarrow\Omega^{0,q+1}(X)$, where $\hat\ddbar_\beta$ is a classical pseudodifferential operator of order $1$ and we have
\begin{equation}\label{e-gue171026}
\mbox{$\hat\ddbar_\beta=\ddbar_b$+lower order terms},
\end{equation}
where $\ddbar_b$ denotes the tangential Cauchy-Riemann operator on $X$. The new operator $\hat\ddbar_{\beta}$ commutes with the $S^1$-action and $\hat\ddbar^2_{\beta}=0$ and hence we can define $\hat H^q_{\beta,m}(X)$ the $m$-th Fourier component of the $q$-th  cohomology group with respect to $\hat\ddbar_{\beta}$-complex. In Section~\ref{s-gue171010}, we develop some kind of reduction to the boundary technique and we can show that for every $q=0,1,\ldots,n-1$,
\begin{equation}\label{e-gue171026I}
H^q_m(\ol M)\cong \hat H^q_{\beta,m}(X),\ \ \forall m\gg1.
\end{equation}

\begin{remark}\label{r-gue200123}
It should be mention that there is no corresponding result for negative frequencies even of large modules. 
Let us see a simple example (see also Example (III) at the end of this section). Assume that the boundary $X$ is strongly pseudoconvex. By Kohn's result (~\cite{FK72}), 
$\Box^{(q)}$ is hypoelliptic, for every $q=1,\ldots,n-1$,  and 
${\rm dim\,}H^q(\ol M)$ is finite dimensional, $q=1,\ldots,n-1$, and hence  ${\rm dim\,}H^q_m(\ol M)=0$, for every $m\in\mathbb Z$, $\abs{m}\gg1$, $q=1,\ldots,n-1$. 
But on the boundary $X$, the Kohn Laplacian for $(0,n-1)$ forms is not hypoelliptic and by a similar argument as in the proofs of \eqref{e-gue171023I} and \cite[Theorem 2.10]{HL15} we have ${\rm dim}\hat H^{n-1}_{\beta, m}(X)\approx |m|^{n-1}$ when $m\rightarrow -\infty$. 
Hence \eqref{e-gue171026I} is not true in general when $m\ll -1$. Indeed, we will show in the end of 
Section~\ref{s-gue200123} that ${\rm dim\,}H^j_m(\ol M)=0$  when $m\ll -1$, for every $j=0,1,\ldots,n-1$ (see Theorem~\ref{t-gue2001}). 

We should also mention that the technical reason for the role of positive Fourier coefficients is Lemma \ref{t-gue171012ta}.
\end{remark}


From \eqref{e-gue171026I}, \eqref{e-gue171026} and the homotopy invariance of the index, we conclude that if $m\gg1$, then
\begin{equation}\label{e-gue171026II}
\sum^{n-1}_{j=0}(-1)^j{\rm dim\,}H^j_m(\ol M)=\sum^{n-1}_{j=0}(-1)^j{\rm dim\,}H^j_{b,m}(X),
\end{equation}
where $H^j_{b,m}(X)$ denotes the $m$-th Fourier component of the $j$-th Kohn-Rossi cohomology group. It was established in~\cite{CHT} an index formula for $\sum^{n-1}_{j=0}(-1)^j{\rm dim\,}H^j_{b,m}(X)$. Combining the result in~\cite{CHT} and \eqref{e-gue171026II}, we obtain an exact formula for $\sum^{n-1}_{j=0}(-1)^j{\rm dim\,}H^j_m(\ol M)$ when $m\gg1$. Note that by using Kohn's estimate, it is easy to see that for $m\gg1$, ${\rm dim\,}H^n_m(\ol M)=0$. Our index formula tells us that for $m\gg1$, $\sum^{n-1}_{j=0}(-1)^j{\rm dim\,}H^j_m(\ol M)$ is an integration of some characteristic forms over the boundary $X$. 
When $\abs{m}$ is small, it might be possible to prove a formula for  $\sum^{n}_{j=0}(-1)^j{\rm dim\,}H^j_m(\ol M)$ involving integration of some characteristic forms over the domain $M$. In Section~\ref{s-gue171021}, we generalize the technique developed in~\cite{HL15} to $\hat\ddbar_\beta$ case (pseudodifferential operator case) and by using the identification \eqref{e-gue171026I}, we successfully establish Morse inequalities for $H^q_m(\ol M)$.

We now formulate our main results. We refer the reader to Section~\ref{s:prelim} for some standard notations and terminology used here.
Let $M$ be a relatively compact open subset with connected smooth boundary $X$ of a complex manifold $M'$ of dimension $n$.
Let $\rho\in C^\infty(M',\Real)$ be a defining function of $X$, that is,
\[X=\{x\in M';\, \rho(x)=0\},\ \ M=\{x\in M';\, \rho(x)<0\}\]
and $d\rho(x)\neq0$ at every point $x\in X$.
Then the manifold $X$ is a CR manifold with a natural CR structure $T^{1,0}X:=T^{1, 0}M'|_X\cap \mathbb CTX$, where $T^{1,0}M'$ denotes the
holomorphic tangent bundle of $M'$.

Assume that $M'$ admits a holomorphic $S^{1}$-action $e^{i\theta}$, $\theta\in[0,2\pi]$, $e^{i\theta}: M'\to M'$, $x\in M'\To e^{i\theta}\circ x\in M'$. It means that
the $S^{1}$-action preserves the complex structure $J$ of $M'$.
In this work, we assume that

\begin{ass}\label{a-gue170929}
The $S^{1}$-action preserves the boundary $X$, that is, we can find a defining function $\rho\in C^\infty(M',\Real)$ of $X$ such that $\rho(e^{i\theta}\circ x)=\rho(x)$, for every $x\in M'$ and every $\theta\in[0,2\pi]$.
\end{ass}

The $S^1$-action $e^{i\theta}$ induces a $S^1$-action $e^{i\theta}$ on $X$. Put
\begin{equation}\label{e-gue171021y}
X_{{\rm reg\,}}:=\set{x\in X;\, e^{i\theta}\circ x\neq x,\ \ \forall\theta\in]0,2\pi[}.
\end{equation}
In this work, we assume that
\begin{equation}\label{e-gue171029}
\mbox{$X_{{\rm reg\,}}$ is non-empty}.
\end{equation}
Since $X$ is connected,  $X_{{\rm reg\,}}$ is an open subset of $X$ and $X\setminus X_{{\rm reg\,}}$ is of measure zero.

Let $T\in C^\infty(M',TM')$ be the global real vector field induced by $e^{i\theta}$, that is $(Tu)(x)=\frac{\pr}{\pr\theta}u(e^{i\theta}\circ x)|_{\theta=0}$, for every $u\in C^\infty(M')$.  In this work, we assume that

\begin{ass}\label{a-gue170929I}
$\Complex T(x)\oplus T^{1,0}_xX\oplus T^{0,1}_xX=\Complex T_xX$, for every $x\in X$.
\end{ass}

The Assumption \ref{a-gue170929I} implies that the $S^1$-action on $M'$ induces a locally free $S^1$-action on $X$. Let $\rho$ be the  defining function of $M$ given in the Assumption \ref{a-gue170929}. Since $X$ is connected and by the Assumption \ref{a-gue170929I},  $\langle J(d\rho), T\rangle$ is always non-zero on $X$. In this work, we always assume that
\begin{equation}\label{ass-111217}
\langle J(d\rho), T\rangle<0 ~\text{on}~X,
\end{equation}
where $J$ is the complex structure tensor on $M'$.

Let $\Omega^{0,q}(M')$ denote the space of smooth $(0,q)$ forms on $M'$ and let $\Omega^{0,q}(\ol M)$ denote the space of restrictions to $M$ of elements in $\Omega^{0,q}(M')$. For every $m\in\mathbb{Z}$, put
\begin{equation}\label{e-gue170929II}
\Omega^{0,q}_{m}(M')=\{u\in\Omega^{0,q}(M');\, \mathcal L_Tu=imu \}
\end{equation}
where $\mathcal L_Tu$ is the Lie derivative of $u$ along direction $T$. For convenience, we write $Tu:=\mathcal L_Tu$. Similarly, let $\Omega^{0,q}_m(\ol M)$ denote the space of restrictions to $M$ of elements in $\Omega^{0,q}_m(M')$. We write $C^\infty(\ol M):=\Omega^{0,0}(\ol M)$, $C^\infty_m(\ol M):=\Omega^{0,0}_m(\ol M)$. Let $\ddbar:\Omega^{0,q}(M')\To\Omega^{0,q+1}(M')$
be the part of the exterior differential operator which maps forms of type $(0,q)$ to forms of
type $(0,q+1)$.
Fix $m\in\mathbb Z$. Since the $S^{1}$-action preserves the complex structure $J$ of $M'$, $T$ commutes with $\ddbar$ and hence
\[\overline\partial: \Omega_m^{0, q}(\overline M)\rightarrow\Omega_m^{0, q+1}(\overline M).\]
We write $\ddbar_m$ to denote the restriction of $\ddbar$ to $\Omega^{0,q}_m(\ol M)$.
We have the following $\ddbar_m$-complex
\begin{equation}\label{complex}
\cdot\cdot\cdot\rightarrow\Omega^{0,q-1}_{m}(\ol M)\xrightarrow{\ddbar_m}\Omega^{0,q}_{m}(\ol M)
\xrightarrow{\ddbar_m}\Omega^{0,q+1}_{m}(\ol M)\rightarrow\cdot\cdot\cdot
\end{equation}
The $m$-th Fourier component of the $q$-th Dolbeault cohomology group is given by
\begin{equation}\label{e-gue170929IId}
H^{q}_{m}(\ol M):=\frac{{\rm Ker\,}\ddbar_m: \Omega_m^{0, q}(\overline M)\rightarrow\Omega_m^{0, q+1}(\overline M)}{{\rm Im\,}\ddbar_m: \Omega_m^{0, q-1}(\overline M)\rightarrow\Omega_m^{0, q}(\overline M)}.
\end{equation}
We will prove in Theorem~\ref{t-gue171008pm} that for every $q=0,1,\ldots,n-1$, and every $m\in\mathbb Z$, we have
\begin{equation}\label{e-gue170929III}
{\rm dim\,}H^{q}_{m}(\ol M)<+\infty.
\end{equation}

We introduce some notations. Take $\omega_0=J(d\rho)$ where $\rho$ is an $S^1$-invariant defining function for $M$ defined in (\ref{e-gue171006aI}). 
{
For $p\in X$, the Levi form $\mathcal{L}_p$ (with respect to $\omega_0$) is the Hermitian quadratic form on $T^{1,0}_pX$ given by
\begin{equation}\label{e-gue170930a}
\mathcal{L}_p(U, V)=-\frac{1}{2i}\langle\,d\omega_0(p)\,,\,U\wedge\ol V\,\rangle,\ \ U, V\in T^{1,0}_pX.
\end{equation}
For ever $j=0,1,\ldots,n-1$, put
\begin{equation}\label{e-gue170930I}
X(j):=\set{x\in X;\, \mbox{$\mathcal{L}_x$ has exactly $j$ negative eigenvalues and $n-1-j$ positive eigenvalues}}.
\end{equation}
\subsection{Main results}

The first main result of this work is the following index formula

\begin{theorem}\label{t-gue171002}
With the notations and assumptions above, there is an $m_0>0$ such that for every $m\in\mathbb Z$ with $m\geq m_0$, we have
\begin{equation}\label{e-gue171002}
\sum_{j=0}^{n-1}(-1)^j{\rm dim\,}H^{j}_{m}(\ol M)=\frac{1}{2\pi}\int_X \mathrm{Td_b\,}(T^{1,0}X) \wedge e^{-m\frac{d\omega_0}{2\pi}} \wedge(-\omega_0),
\end{equation}
where $\mathrm{Td_b\,}(T^{1,0}X)$ denotes the \emph{tangential Todd class} of $T^{1,0}X$ (see Definition~\ref{d-gue150516}).
\end{theorem}

By using Kohn's estimate, it is easy to see that $H^n_m(\ol M)=\set{0}$ for $m\gg1$. Hence in the index formula \eqref{e-gue171002}, we only need to sum over $j=0,1,\ldots,n-1$. Our second main result is the following weak and strong Morse inequalities.

\begin{theorem}\label{t-gue170930}
With the notations and assumptions above, for every $q=0,1,\ldots,n-1$, we have
\begin{equation}\label{e-gue170930II}
{\rm dim\,}H^q_{m}(\ol M)\leq \frac{m^{n-1}}{(n-1)!(2\pi)^n}\int_{X(q)}(-1)^{n+q}(d\omega_0)^{n-1}\wedge\omega_0+o(m^{n-1}),\ \ m\in\mathbb N,
\end{equation}
and
\begin{equation}\label{e-gue170930III}
\begin{split}
&\sum_{j=0}^q(-1)^{q-j}{\rm dim\,}H^j_m(\overline M)\\
&\leq\frac{m^{n-1}}{(n-1)!(2\pi)^n}\sum_{j=0}^q(-1)^{q-j}\int_{X(j)}(-1)^{n+j}(d\omega_0)^{n-1}\wedge\omega_0+o(m^{n-1}),\ \ m\in\mathbb N.
\end{split}
\end{equation}
\end{theorem}

In Berman's work~\cite{Be05}, he considered a holomorphic line bundle $L$ over a complex manifold with boundary and he studied large $k$ behaviour of the $q$-th Dolbeault cohomology group with values of $L^k$ under condition $Z(q)$. He generalized Demailly's holomorphic Morse inequalities  to complex manifolds with boundary under condition $Z(q)$. Here condition $Z(q)$ is a sufficient condition such that the $q$-th Dolbeault cohomology group is finite dimensional. In this work, we do not consider line bundles but we study circle equivariant holomorphic Morse inequalities and 
we do not need condition $Z(q)$ since the Fourier components $H^q_m(\overline M)$ of $q$-th Dolbeault cohomology are always finite dimensional (see Theorem \ref{t-gue171008pm}). It is also quite interesting to study circle equivariant version of Berman's work.

When $q=1$, from Theorem ~\ref{t-gue170930} we have
\begin{equation}\label{e1-171109}
\begin{split}
&-{\rm dim\,}H^0_m(\overline M)+{\rm dim\,}H^1_m(\overline M)\leq\\ &\frac{m^{n-1}}{(n-1)!(2\pi)^n}\left(-\int_{X(0)}(-1)^{n}(d\omega_0)^{n-1}\wedge\omega_0
+\int_{X(1)}(-1)^{n+1}(d\omega_0)^{n-1}\wedge\omega_0\right)+o(m^{n-1}).
\end{split}
\end{equation}

\begin{theorem}\label{t-gue180817}
With the notations and assumptions above and we assume that $M$ is a weakly pseudoconvex manifold and strongly pseudoconvex at a point. Then ${\rm dim}~H^0_{m}(\overline M)\approx m^{n-1}$  as $m\rightarrow\infty$. Moreover, ${\rm dim}H^0(\overline M)=\infty.$
\end{theorem}

For the definition of weakly pseudoconvex manifold please turn to Definition \ref{def-111217}.
If we set $$\int_{X(\leq 1)}(-1)^n(d\omega_0)^{n-1}\wedge\omega_0=\int_{X(0)}(-1)^n(d\omega_0)^{n-1}\wedge\omega_0
-\int_{X(1)}(-1)^n(d\omega_0)^{n-1}\wedge\omega_0
$$ and assume that $M$ is not always weakly pseudoconvex but that the integral $$\int_{X(\leq 1)}(-1)^n(d\omega_0)^{n-1}\wedge \omega_0>0,$$ then by (\ref{e1-171109}) we still get many holomorphic functions on $M$ which are smooth up to the boundary.

\begin{theorem}
With the same notations and assumptions as in Theorem ~\ref{t-gue170930} and we assume that
\begin{equation}
\int_{X(\leq 1)}(-1)^n(d\omega_0)^{n-1}\wedge \omega_0>0.
\end{equation}
Then ${\rm dim}H^0_m(\overline M)\approx m^{n-1}$. Moreover, ${\rm dim}H^0(\overline M)=\infty.$
\end{theorem}

Theorem~\ref{t-gue180817} is closely related to the classical Levi problem in several complex variables. Let us recall some background about the classical Levi problem.
 Let $M\subset \mathbb C^n$ be a relatively compact domain with connected smooth boundary.
In 1911, Levi \cite{Le11} proved that if $M$ is a holomorphic convex domain, then $M$ is a weakly pseudoconvex domain.  Then Levi conjectured that the converse is true.
The Levi conjecture is confirmed by Oka \cite{Oka42} in $\mathbb C^2$, Oka \cite{Oka53},  Bremermann \cite{Br54}, Norguet \cite{No54} in general $\mathbb C^n$, $n>2$. We can also ask similar question on general complex manifolds.  Grauert \cite{Gr58} proved that the Levi conjecture is still true when $M$ is a strongly pseudoconvex complex manifold.
In 1958, Grauert gave a counterexample to show that the Levi conjecture need not to be true for weakly pseudoconvex domain in general complex manifolds. Then, Narasimhan \cite{Na62} asked that if $M$ is weakly pseudoconvex and strongly pseudoconvex at point, is $M$ holomorphic convex?
Grauert \cite{Gr63} constructed an example to give a negative answer to Narasimhan's problem. Later, Michel \cite{Mi76}  gave an affirmative answer to Narasimhan's question in the case of a compact homogeneous manifold.
Narasimhan's problem plays an important role in several complex variables and is closely related to the following problem: on what kind of weakly pseudoconvex complex manifold $M$ can we have ${\rm dim\,}H^0(\ol M)=\infty$? Theorem~\ref{t-gue180817} offers a class of weakly pseudoconvex complex manifolds  such that ${\rm dim\,}H^0(\ol M)=\infty$.

Let $E$ be a $S^1$-invariant holomorphic vector bundle over $M'$. As \eqref{e-gue170929IId}, we can define $H^q_m(\ol M,E)$ the $m$-th Fourier component of
the $q$-th Dolbeault cohomology group with values in $E$. We can repeat the proofs of Theorem~\ref{t-gue171002} and Theorem~\ref{t-gue170930} with minor changes and deduce

\begin{theorem}\label{t-gue171027}
With the notations and assumptions above, let $E$ be an $S^1$-invariant holomorphic vector bundle over $M'$ of complex rank $r$. There is an $m_0>0$ such that for every $m\in\mathbb Z$ with $m\geq m_0$, we have
\begin{equation}\label{e-gue171027}
\sum_{j=0}^{n-1}(-1)^j{\rm dim\,}H^{j}_{m}(\ol M, E)=\frac{1}{2\pi}\int_X \mathrm{Td_b\,}(T^{1,0}X) \wedge \mathrm{ch_b\,}(E)\wedge e^{-m\frac{d\omega_0}{2\pi}} \wedge(-\omega_0),
\end{equation}
where $\mathrm{ch_b\,}(E)$ denotes the tangential Chern character of
$E$ (see Definition~\ref{d-gue150516}).
\end{theorem}

\begin{theorem}\label{t-gue171027I}
With the notations and assumptions above, let $E$ be an $S^1$-invariant holomorphic vector bundle over $M'$ of complex rank $r$. For every $q=0,1,\ldots,n-1$, we have
\begin{equation}\label{e-gue171027I}
{\rm dim\,}H^q_{m}(\ol M,E)\leq r\frac{m^{n-1}}{(n-1)!(2\pi)^n}\int_{X(q)}(-1)^{n+q}(d\omega_0)^{n-1}\wedge\omega_0+o(m^{n-1}),\ \ m\in\mathbb N,
\end{equation}
and
\begin{equation}\label{e-gue171027II}
\begin{split}
&\sum_{j=0}^q(-1)^{q-j}{\rm dim\,}H^j_m(\overline M,E)\\
&\leq r\frac{m^{n-1}}{(n-1)!(2\pi)^n}\sum_{j=0}^q(-1)^{q-j}\int_{X(j)}(-1)^{n+j}(d\omega_0)^{n-1}\wedge\omega_0+o(m^{n-1}),\ \ m\in\mathbb N.
\end{split}
\end{equation}
\end{theorem}

\begin{corollary}
With the same notations and assumptions as in Theorem~\ref{t-gue171027I} and we assume that the integral
\begin{equation}
\int_{X(\leq 1)}(-1)^n(d\omega_0)^{n-1}\wedge \omega_0>0.
\end{equation}
Then ${\rm dim\,}H^0_m(\overline M, E)\approx rm^{n-1}$ as $m\rightarrow\infty$. Moreover, ${\rm dim\,}H^0(\overline M, E)=\infty.$
\end{corollary}


For the case when $m<0$, we have the following
\begin{theorem}\label{t-gue2001}
With the notations and assumptions used in Theorem~\ref{t-gue171002}, there is an
$m_1\in\mathbb{N}$ such that for every $m\in\mathbb{Z}, m<0, |m|>m_1$, we have
$\dim H_m^q(\ol M)=0$, for every $q=0,1,...,n-1$.
\end{theorem}

In the end of this section, we give some simple examples.

(I)\, Consider a compact complex manifold $\Omega$ and let $
(L,h^{L})\rightarrow \Omega$ be a holomorphic line bundle over $\Omega$, where $h^{L}$
denotes a Hermitian fiber metric of $L$. Let $(L^{\ast },h^{L^{\ast
}})\rightarrow\Omega$ be the dual bundle of $(L,h^{L})$ and put $M'=L^\ast$ and let $M=\left\{ v\in
L^{\ast };\,\left\vert v\right\vert _{h^{L^{\ast }}}^{2}<1\right\} $.
Clearly $M'$ is
equipped with a natural (globally free) $S^1$-action $e^{i\varphi}$ (by acting on the circular fiber) and the $S^1$-action satisfies Assumption~\ref{a-gue170929}, Assumption~\ref{a-gue170929I} and \eqref{e-gue171029}.

Assume that $\Omega$ admits a holomorphic torus action \(T^d\curvearrowright\Omega\) denoted by $(e^{i\varphi_1},\ldots,e^{i\varphi_d})$.
The torus action $T^d$ and the $S^1$-action $e^{i\varphi}$ induce a natural $S^1$-action on $M'$ given by $e^{i\theta}\circ (z,\lambda)=((e^{i\varphi_1},\ldots,e^{i\varphi_d})\circ z, e^{i\varphi}\circ \lambda)$, where $\varphi=\varphi_1=\cdots=\varphi_d=\theta$, $z$ denotes the coordinates of $\Omega$ and $\lambda$ denotes the fiber coordinate of $L^\ast$. It is easy to see that this $S^1$-action satisfies Assumption~\ref{a-gue170929}, Assumption~\ref{a-gue170929I} and \eqref{e-gue171029}. 

Since $M'$ admits a holomorphic torus action, it is a natural question to study the space of torus equivariant holomorphic sections 
\[\begin{split}
&H^0_{\mu_1,\ldots,\mu_d,\mu_{d+1}}(\ol M)=\{u\in H^0(\ol M);\, (e^{i\varphi_1},\ldots,e^{i\varphi_{d+1}})^*u=e^{i\mu_1\varphi_1+\cdots+i\mu_{d+1}\varphi_{d+1}}u,\\ 
&\quad\mbox{for every $(e^{i\varphi_1},\ldots,e^{i\varphi_d},e^{i\varphi_{d+1}})\in T^d\times S^1$}\},\end{split}\]
where $(\mu_1,\ldots,\mu_d,\mu_{d+1})\in\mathbb Z^{d+1}$. To study such kind of question, one usually needs to introduce momentum map associated to the torus action 
and one needs more conditions about the torus action. On the other hand, if we looks at the diagonal action as above, the space $H^0_m(\ol M)$ is equal to the space 
\[\oplus_{(\mu_1,\ldots,\mu_d,\mu_{d+1})\in\mathbb Z^{d+1}, \mu_1+\cdots+\mu_{d+1}=m}H^0_{\mu_1,\ldots,\mu_d,\mu_{d+1}}(\ol M)\]
and we can get many  torus equivariant holomorphic sections if the space $H^0_m(\ol M)$ is large. Moreover, we can modify the diagonal action: 
Let $p_1\in\mathbb N,\ldots,p_{d+1}\in\mathbb N$ and assume that the greatest common divisor of $(p_1,\ldots,p_{d+1})$ is one. Consider the modified diagonal $S^1$ action:
\[e^{i\theta}\circ (z,\lambda)=((e^{ip_1\theta},\ldots,e^{ip_d\theta})\circ z, e^{ip_{d+1}\theta}\circ \lambda).\]
The asymptotic behaviour of the space $H^0_m(\ol M)$ with respect to the modified diagonal $S^1$ action 
has some interesting applications in number theory (see Example (III) below). 

(II)\,  Let $V\subset\mathbb C^m$ be a complex space with an isolated singularity $0\in V$. 
Let $e^{i\theta}$ be the standard $S^1$-action on $\mathbb C^m$ given by 
$e^{i\theta}\circ(z_1, \ldots, z_m):=(e^{i\theta}z_1, \ldots, e^{i\theta}z_m)$, for every $e^{i\theta}\in S^1$ and every $(z_1,\ldots,z_m)\in\mathbb C^m$. 
Assume that $V$ is symmetric with respect to the standard $S^1$-action $e^{i\theta}$ on $\mathbb C^m$. That is, if $(z_1, \ldots, z_m)\in V$, 
then $e^{i\theta}\circ(z_1, \cdots, z_m)=(e^{i\theta}z_1, \ldots, e^{i\theta}z_m)\in V$, for every $e^{i\theta}\in S^1$. Then, the $S^1$-action $e^{i\theta}$ acts on $V$. We suppose further that $0$ is the unique fixed point of the $S^1$-action $e^{i\theta}$.
Let $\mathbb B_r\subset\mathbb C^m$ be a ball centered at $0$ with radius $r$. Let $V_r=V\cap \mathbb B_r$. Suppose $\mathbb B_r$ intersects $V$ transversally. By the Sard theorem, for almost all $0<r\ll1$, the boundary of $V_r$ denoted by $X_r$ is a smooth strongly pseudoconvex CR manifold. The $X_r's$ are called CR links of the germ $(V, 0)$. It is an useful approach to study the singularity through studying its CR links (see \cite{H06}). Assume  the  holomorphic $S^1$-action on $V$ induces  a transversal  $S^1$-action on $X_r$.  Let $\pi: \tilde V_r\rightarrow V_r$ be a desingularization. The $S^1$ action on $V_r$ lifts to a holomorphic $S^1$ action on $\tilde V_r$. Thus, $\tilde V_r$ is a complex manifold which admits a holomorphic $S^1$ action which preserves the boundary $X_r$ and the action restricted on the boundary of $\tilde V_r$ is transversal and CR.

(III) Let $M=\mathbb B^N:=\set{(z_1,\ldots,z_N)\in\mathbb C^N;\, \abs{z_1}^2+\cdots+\abs{z_N}^2<1}$, $N\geq2$.  
Let $p_1\in\mathbb N,\ldots,p_{N}\in\mathbb N$. Assume that the greatest common divisor of $(p_1,\ldots,p_N)$ is equal to one. The $S^1$-action on $M$ is given by 
\[e^{i\theta}\circ(z_1,\ldots,z_N)=(e^{ip_1\theta}z_1, \ldots,e^{ip_N\theta}z_N).\] 
Then it is easy to see that for every $m\in\mathbb Z$,
\begin{equation*}
\begin{split}
H^0_m(\ol M)&=\text{Span}\{z_1^{\alpha_1}\cdots z^{\alpha_N}_N;\, p_1\alpha_1+\cdots+p_{N}\alpha_{N}=m, (\alpha_1,\ldots,\alpha_N)\in\mathbb N^N_0  \},\, m\geq 0,\\
H^0_m(\ol M)&=0, \forall m\in\mathbb{Z}, m<0.
\end{split}
\end{equation*}
Since the boundary of $M$ is strongly pseudoconvex, $H^q(\ol M)$ is finite dimensional, for every $q=1,\ldots,n-1$. 
As a consequence $H_m^q(\overline M)=0$ for all $m\in\mathbb Z$, $\abs{m}\gg1$, $q=1,\ldots,n-1$. However, when $m\ll -1$, by a similar argument as in the proofs of \eqref{e-gue171023I} and \cite[Theorem 2.10]{HL15} we have ${\rm dim}\hat H^{n-1}_{\beta, m}(X)\approx |m|^{N-1}$ when $m\rightarrow -\infty$. Hence \eqref{e-gue171026I} is not true in general when $m\ll -1$.  Moreover, the dimension of $H^0_m(\ol M)$ is equal to the number of non-negative integral solutions $(x_1,\ldots,x_N)$ of 
the equation
\[x_1p_1+\cdots+x_{N}p_{N}=m.\]
Then, from our main result (see Theorem~\ref{t-gue180817}), we see that the number of non-negative integral solutions $(x_1,\ldots,x_N)$ of the equation above is $\approx m^{N-1}$ if $m\gg1$. Furthermore, from the index formula Theorem~\ref{t-gue171002} and vanishing results for $H_m^q(\overline M)$, 
$q=1,\ldots,n-1$, $m\To+\infty$, we can get precise formula for the number of non-negative integral solutions $(x_1,\ldots,x_N)$ of the equation above when $m\gg1$.


\section{Preliminaries}\label{s:prelim}

\subsection{Some standard notations}\label{s-gue150508b}
We use the following notations: $\mathbb N=\set{1,2,\ldots}$,
$\mathbb N_0=\mathbb N\cup\set{0}$, $\Real$
is the set of real numbers, $\ol\Real_+:=\set{x\in\Real;\, x\geq0}$.
For a multiindex $\alpha=(\alpha_1,\ldots,\alpha_m)\in\mathbb N_0^m$,
we set $\abs{\alpha}=\alpha_1+\cdots+\alpha_m$. For $x=(x_1,\ldots,x_m)\in\Real^m$ we write
\[
\begin{split}
&x^\alpha=x_1^{\alpha_1}\ldots x^{\alpha_m}_m,\quad
 \pr_{x_j}=\frac{\pr}{\pr x_j}\,,\quad
\pr^\alpha_x=\pr^{\alpha_1}_{x_1}\ldots\pr^{\alpha_m}_{x_m}=\frac{\pr^{\abs{\alpha}}}{\pr x^\alpha}\,,\\
&D_{x_j}=\frac{1}{i}\pr_{x_j}\,,\quad D^\alpha_x=D^{\alpha_1}_{x_1}\ldots D^{\alpha_m}_{x_m}\,,
\quad D_x=\frac{1}{i}\pr_x\,.
\end{split}
\]
Let $z=(z_1,\ldots,z_m)$, $z_j=x_{2j-1}+ix_{2j}$, $j=1,\ldots,m$, be coordinates of $\Complex^m$,
where
$x=(x_1,\ldots,x_{2m})\in\Real^{2m}$ are coordinates in $\Real^{2m}$.
We write
\[
\begin{split}
&z^\alpha=z_1^{\alpha_1}\ldots z^{\alpha_m}_m\,,\quad\ol z^\alpha=\ol z_1^{\alpha_1}\ldots\ol z^{\alpha_m}_m\,,\\
&\pr_{z_j}=\frac{\pr}{\pr z_j}=
\frac{1}{2}\Big(\frac{\pr}{\pr x_{2j-1}}-i\frac{\pr}{\pr x_{2j}}\Big)\,,\quad\pr_{\ol z_j}=
\frac{\pr}{\pr\ol z_j}=\frac{1}{2}\Big(\frac{\pr}{\pr x_{2j-1}}+i\frac{\pr}{\pr x_{2j}}\Big),\\
&\pr^\alpha_z=\pr^{\alpha_1}_{z_1}\ldots\pr^{\alpha_m}_{z_m}=\frac{\pr^{\abs{\alpha}}}{\pr z^\alpha}\,,\quad
\pr^\alpha_{\ol z}=\pr^{\alpha_1}_{\ol z_1}\ldots\pr^{\alpha_m}_{\ol z_m}=
\frac{\pr^{\abs{\alpha}}}{\pr\ol z^\alpha}\,.
\end{split}
\]

Let $\Omega$ be a $C^\infty$ orientable paracompact manifold.
We let $T\Omega$ and $T^*\Omega$ denote the tangent bundle of $\Omega$ and the cotangent bundle of $\Omega$ respectively.
The complexified tangent bundle of $\Omega$ and the complexified cotangent bundle of $\Omega$
will be denoted by $\Complex T\Omega$ and $\Complex T^*\Omega$ respectively. We write $\langle\,\cdot\,,\cdot\,\rangle$
to denote the pointwise duality between $T\Omega$ and $T^*\Omega$.
We extend $\langle\,\cdot\,,\cdot\,\rangle$ bilinearly to $\Complex T\Omega\times\Complex T^*\Omega$.

Let $E$ be a $C^\infty$ complex vector bundle over $\Omega$. The fiber of $E$ at $x\in\Omega$ will be denoted by $E_x$.
Let $F$ be another vector bundle over $\Omega$. We write
$F\boxtimes E^*$ to denote the vector bundle over $\Omega\times\Omega$ with fiber over $(x, y)\in\Omega\times\Omega$
consisting of the linear maps from $E_y$ to $F_x$.

Let $Y\subset \Omega$ be an open set. The spaces of smooth sections of $E$ over $Y$ and distribution sections of $E$ over $Y$ will be denoted by $C^\infty(Y, E)$ and $\mathscr D'(Y, E)$ respectively.
Let $\mathscr E'(Y, E)$ be the subspace of $\mathscr D'(Y, E)$ whose elements have compact support in $Y$ and set 
$C^\infty_0(Y,E):=C^\infty(Y,E)\bigcap\mathscr E'(Y, E)$. Fix a volume form on $Y$ and a Hermitian metric on $E$, we get a 
natural $L^2$ inner product $(\,\cdot\,|\,\cdot\,)$ on $C^\infty_0(Y,E)$. Let $L^2(Y,E)$ be the completion of $C^\infty_0(Y,E)$ with respect 
to $(\,\cdot\,|\,\cdot\,)$ and the $L^2$ inner product $(\,\cdot\,|\,\cdot\,)$ can be extended to $L^2(Y,E)$ by density. Let $\norm{\cdot}$ be the $L^2$ norm corresponding to the $L^2$ inner product $(\,\cdot\,|\,\cdot\,)$. 
For every $s\in\mathbb R$, let $L_s: \mathscr D'(Y,E)\To\mathscr D'(Y,E)$ be a properly supported classical 
elliptic pseudodifferential operator of order $s$ on $Y$ with values in $E$. Define 
\[W^s(Y,E):=\set{u\in \mathscr D'(Y,E);\, L_s u\in L^2(Y,E)}\]
and for $u\in W^s(Y,E)$, let $\norm{u}_s:=\norm{L_s u}$. 
We call $W^s(Y, E)$ the Sobolev space of order $s$ of sections of $E$ over $Y$ (with respect to $L_s$) and for $u\in W^s(Y,E)$, we call the number $\norm{u}_s$ the Sobolev norm of $u$ of order $s$ (with respect to $L$).
Put
\begin{gather*}
W^s_{\rm loc\,}(Y, E)=\big\{u\in\mathscr D'(Y, E);\, \varphi u\in W^s(Y, E),
     \,\forall\varphi\in C^\infty_0(Y)\big\}\,,\\
      W^s_{\rm comp\,}(Y, E)=W^s_{\rm loc}(Y, E)\cap\mathscr E'(Y, E)\,.
\end{gather*}



\subsection{Set up}\label{s-gue171006q}

Let $M$ be a relatively compact open subset with smooth connected boundary $X$ of a complex manifold $M'$ of dimension $n$. Assume that $M'$ admits a holomorphic $S^{1}$-action $e^{i\theta}$, $\theta\in[0,2\pi]$: $e^{i\theta}: M'\to M'$, $x\in M'\To e^{i\theta}\circ x\in M'$. From now on, we will use the same assumptions and notations as in Section~\ref{s-gue170929}. Recall that we work with Assumption~\ref{a-gue170929}, Assumption~\ref{a-gue170929I} and \eqref{e-gue171029}. From now on, we fix a $S^1$-invariant Hermitian metric $\langle\,\cdot\,|\,\cdot\,\rangle$ on $\Complex TM'$ so that
\begin{equation}\label{e-gue171006a}
\begin{split}
&\mbox{$T\perp(T^{1, 0}X\oplus T^{0, 1}X)$ at every point of $X$},\\
&\mbox{$\langle\,T\,|\,T\,\rangle=1$ at every point of $X$}.
\end{split}
\end{equation}
The Hermitian metric $\langle\,\cdot\,|\,\cdot\,\rangle$ on $\Complex TM'$ induces by duality a Hermitian metric on $\Complex T^*M'$ and Hermitian metrics on $T^{*0,q}M'$ the bundle of $(0,q)$ forms on $M'$, $q=1,\ldots,n$. We shall also denote these Hermitian metrics by $\langle\,\cdot\,|\,\cdot\,\rangle$.
From now on, we fix a defining function $\rho\in C^\infty(M',\Real)$ of $X$ such that
\begin{equation}\label{e-gue171006aI}
\begin{split}
&\mbox{$\langle\,d\rho(x)\,|\,d\rho(x)\,\rangle=1$ on $X$},\\
&\rho(e^{i\theta}\circ x)=\rho(x),\ \ \forall x\in M',\ \ \forall\theta\in [0,2\pi[.
\end{split}\end{equation}
From the assumption (\ref{ass-111217}),  (\ref{e-gue171006a}) and (\ref{e-gue171006aI})  we have
\begin{equation}
T=J(\frac{\partial}{\partial \rho})~\text{on} ~X,
\end{equation}
where $\frac{\partial}{\partial\rho}$ is the dual vector field of $d\rho$ with respect to the given Hermitian metric on $M'$. 


\begin{definition}\label{def-111217}
$M$ is called weakly (strongly) pseudoconvex at $x\in X$ if $\mathcal L_x$ is positive semidefinite (definite) on $T_x^{1, 0}X$. If $\mathcal L_x$ is positive semidefinite at every point of $X$, then $M$ is called a weakly pseudoconvex  manifold.
\end{definition}

Fix $q=0,1,\ldots,n-1$. Let $\Omega^{0,q}_0(M')$  be the space of smooth $(0,q)$ forms on $M'$ whose elements have compact support in $M'$ and let $\Omega^{0,q}_0(M)$  be the space of smooth $(0,q)$ forms on $M$ whose elements have compact support in $M$.
Let $(\,\cdot\,|\,\cdot\,)_{M'}$ and $(\,\cdot\,|\,\cdot\,)_{M}$ be the $L^2$ inner products on $\Omega^{0,q}_0(M')$ and $\Omega^{0,q}_0(M)$ respectively given by
\begin{equation}\label{e-gue171006aIb}
\begin{split}
&(\,u\,|\,v\,)_{M'}:=\int_{M'}\langle\,u\,|\,v\,\rangle dv_{M'},\ \ u, v\in\Omega^{0,q}_0(M'),\\
&(\,u\,|\,v\,)_M:=\int_M\langle\,u\,|\,v\,\rangle dv_{M'},\ \ u, v\in\Omega^{0,q}_0(M),
\end{split}
\end{equation}
where $dv_{M'}$ is the volume form on $M'$ induced by $\langle\,\cdot\,|\,\cdot\,\rangle$. Let $L^2_{(0,q)}(M)$ be the $L^2$ completion of $\Omega^{0,q}_0(M)$ with respect to
$(\,\cdot\,|\,\cdot\,)_M$. It is clear that $\Omega^{0,q}(\ol M)\subset L^2_{(0,q)}(M)$. We extend $(\,\cdot\,|\,\cdot\,)_M$ to $L^2_{(0,q)}(M)$ in the standard way and let $\norm{\cdot}_{M}$ be the corresponding $L^2$ norm. Let $T^{*0,q}X$ be the bundle of $(0,q)$ forms on $X$. Recall that for every $x\in X$, we have
\[T^{*0,q}_xX:=\set{u\in T^{*0,q}_xM';\, \langle\,u\,|\,\ddbar\rho(x)\wedge g\,\rangle=0,\ \ \forall g\in T^{*0,q-1}_xM'}.\]
Let $\Omega^{0,q}(X)$ be the space of smooth $(0,q)$ forms on $X$.
Let $(\,\cdot\,|\,\cdot\,)_X$ be the $L^2$ inner product on $\Omega^{0,q}(X)$ given by
\begin{equation}\label{e-gue171006aIc}
(\,u\,|\,v\,)_X:=\int_X\langle\,u\,|\,v\,\rangle dv_X,
\end{equation}
where $dv_X$ is the volume form on $X$ induced by $\langle\,\cdot\,|\,\cdot\,\rangle$. Let $L^2_{(0,q)}(X)$ be the $L^2$ completion of $\Omega^{0,q}(X)$ with respect to
$(\,\cdot\,|\,\cdot\,)_X$. We extend $(\,\cdot\,|\,\cdot\,)_X$ to $L^2_{(0,q)}(X)$ in the standard way and let $\norm{\cdot}_X$ be the corresponding $L^2$ norm.

Fix $m\in\mathbb Z$. Let $\Omega^{0,q}_{m}(M')$ be as in \eqref{e-gue170929II} and let $\Omega^{0,q}_m(\ol M)$ denote the space of restrictions to $M$ of elements in $\Omega^{0,q}_m(M')$. Let $L^2_{(0,q),m}(M)$ be the completion of $\Omega^{0,q}_{m}(\ol M)$ with respect to $(\,\cdot\,|\,\cdot\,)_M$. Similarly, let
\begin{equation}\label{e-gue171006yc}
\Omega^{0,q}_m(X):=\{u\in\Omega^{0,q}(X);\, \mathcal L_Tu=imu\},
\end{equation}
where  $\mathcal L_Tu$ is the Lie derivative of $u$ along direction $T$. For convenience, we write $Tu:=\mathcal L_Tu$. Let $L^2_{(0,q),m}(X)$ be the completion of $\Omega^{0,q}_m(X)$ with respect to $(\,\cdot\,|\,\cdot\,)_X$.

Let $A$ be a $C^\infty$ complex vector bundle over $M'$. For every  $s\in\mathbb R$, fix $L_s: \mathscr D'(M',A)\To\mathscr D'(M',A)$  a properly supported classical 
elliptic pseudodifferential operator of order $s$ on $M'$ with values in $A$ and we define $W^s(M',A)$ the Sobolev space of order $s$ with values in $A$ over $M'$ (with respect to $L_s$) and for a $u\in W^s(M',A)$,  we define $\norm{u}_{s,M'}$ the Sobolev norm of $u$ of 
order $s$ (with respect to $L_s$ ) as in the discussion at the end of Section~\ref{s-gue150508b}. For every $s\in\mathbb R$, let 
\[
\begin{split}
&C^\infty(\ol M,A),\ \ \mathscr D'(\ol M,A),\\
W^s&(\ol M,A),\ \ H^s_{{\rm comp\,}}(\ol M,A),\ \ 
W^s_{{\rm loc\,}}(\ol M,A),
\end{split}
\]
(where $\ s\in\mathbb R$)
denote the spaces of restrictions to $\ol M$ of elements in 
\[
\begin{split}
C^\infty&(M',A),\ \ \mathscr D'(M',A),\\
&W^s(M',A),\ \  W^s_{{\rm comp\,}}(M',A),\ \  
W^s_{{\rm loc\,}}(M',A),
\end{split}
\] 
respectively.   
Let  $u\in W^s(\ol M,A)$. We define
\[\norm{u}_{s,\ol M}:=\inf\set{\norm{\Td u}_{s,M'};\, u'\in W^s(M',A), u'|_M=u}.\]
We call $\norm{u}_{s,\ol M}$ the Sobolev norm of $u$ of order $s$ on $\ol M$. 

Let $s$ be a non-negative integer. We can also define  Sobolev norm of order $s$ on $\ol M$ as follows: Let $x_0\in X$ and let $U$ be an open neighborhood of $x_0$ in $M'$ with local coordinates $x=(x_1,\ldots,x_{2n})$. 
Let $u\in\mathscr E'(U)\bigcap W^s(\ol M)$. Let $\Td u\in \mathscr E'(U)\bigcap W^s(M')$ with $\Td u|_M=u$. We define the Sobolev norm of order $s$ of $u$ on $\ol M$ by
\begin{equation}\label{e-gue171007hyc}
\norm{u}^2_{(s),\ol M}:=\sum_{\alpha\in\mathbb N_0^{2n},\abs{\alpha}\leq s}\int_{M}\abs{\pr^\alpha_x\Td u}^2dv_{M'}.
\end{equation}
By using partition of unity, for $u\in W^s(\ol M)$, we define $\norm{u}^2_{(s),\ol M}$ in the standard way. As in function case, we define $\norm{u}^2_{(s),\ol M}$, for $u\in W^s(\ol M,A)$ in the similar way. 
It is well-known (see~\cite[Corollary B.2.6]{Hor85}) that the two norms $\norm{\cdot}_{s,\ol M}$ and $\norm{\cdot}_{(s),\ol M}$ are equivalent for every non-negative integer $s$.

\section{$S^1$-equivariant $\overline\partial$-Neumann problem}\label{s-gue171006}

In this section, we introduce the $S^1$-equivariant $\overline\partial$-Neumann problem and $S^1$-equivariant Hodge isomorphism (see Theorem \ref{t-gue171008pm}). We leave the details of the proofs of the main results in this section to the Appendix. Until further notice, we fix $m\in\mathbb Z$ and $q\in\set{0,1,\ldots,n-1}$. Let $\ddbar:\Omega^{0,q}(\ol M)\To\Omega^{0,q+1}(\ol M)$ be the Cauchy-Riemann operator.  We write $\overline\partial_m$ to denote the restriction of $\overline\partial$ on $\Omega^{0, q}_m(\overline M)$. Since $T$ commutes with $\ddbar$, we have $\overline\partial_m: \Omega_m^{0, q}(\overline M)\rightarrow\Omega_m^{0, q+1}(\overline M)$.
We extend $\ddbar_m$ to $L^2_{(0,q),m}(M)$:
\[
\ddbar_m:{\rm Dom\,}\ddbar_m\subset L^{2}_{(0,q),m}(M)\rightarrow L^{2}_{(0,q+1),m}(M),
\]
where ${\rm Dom\,}\ddbar_m=\{u\in L^{2}_{(0, q), m}(M);\, \ddbar u\in L^{2}_{(0, q+1), m}(M)\}$.
Let
\[\ddbar^\star_m:{\rm Dom\,}\ddbar^\star_m\subset L^2_{(0,q+1),m}(M)\To L^2_{(0,q),m}(M)\]
be the Hilbert adjoint of $\ddbar_m$ with respect to $(\,\cdot\,|\,\cdot\,)_M$. The Gaffney extension of $m$-th $\ddbar$-Neumann Laplacian is given by
\begin{equation}\label{e-gue171006g}
\Box^{(q)}_m: {\rm Dom\,}\Box^{(q)}_m\subset L^2_{(0,q),m}(M)\To L^2_{(0,q),m}(M),
\end{equation}
where ${\rm Dom\,}\Box^{(q)}_{m}:=\{u\in L^{2}_{(0,q),m}(M): u\in{\rm Dom\,}\overline\partial_m\cap{\rm Dom\,}\overline\partial_m^\star, \ddbar_m u\in{\rm Dom\,}\ddbar_m^{\star}, \ddbar_m^{\star}u\in{\rm Dom\,}\ddbar_m \}$ and $\Box^{(q)}_mu=(\ddbar_m\,\ddbar^{\star}_m+\ddbar^\star_m\,\ddbar_m)u$, $u\in{\rm Dom\,}\Box^{(q)}_m$.
Put
\[{\rm Ker\,}\Box^{(q)}_m=\set{u\in{\rm Dom\,}\Box^{(q)}_m;\, \Box^{(q)}_mu=0}.\]
It is easy to check that
\begin{equation}\label{e-gue171006ycb}
{\rm Ker\,}\Box^{(q)}_{m}=\{u\in{\rm Dom\,}\Box^{(q)}_{m};\, \ddbar_m u=0, \ddbar_m^{\star} u=0\}.
\end{equation}
Exploiting the transversality of the action, we can use integration by parts and Stoke's theorem (see Lemma~\ref{main theorem1-170913}) and get the following estimate: For every $m\in\mathbb Z$, 
there exists $C_{m}>0$ such that for every $u\in{\rm Dom\,}\Box^{(q)}_m\cap\Omega^{0, q}_m(\overline M)$, we have
\begin{equation}\label{e-gue200125yyd}
\|u\|_{1,\ol M}\leq C_{m}(\|\Box^{(q)}_mu\|_{M}+\|u\|_{M}),
\end{equation}
where $\norm{\cdot}_{1,\ol M}$ denotes the Sobolev norm of order $1$ on $\ol M$ (see the discussion before \eqref{e-gue171007hyc}). 
From \eqref{e-gue200125yyd}, the method of  Folland-Kohn (see Section 5 in~\cite{FK72}) and some standard arguments in functional analysis, we get the following

	

\begin{theorem}\label{t-gue171008pm}
The operator $\Box^{(q)}_{m}: {\rm Dom\,}\Box^{(q)}_m\To L^2_{(0,q),m}(M)$ has $L^{2}$ closed range.
	Moreover,
	${\rm Ker\,}\Box^{(q)}_{m}\subset \Omega^{0.q}_{m}(\ol M)$, $\dim{\rm Ker\,}\Box^{(q)}_{m}<\infty$ and 
	\begin{equation}\label{iso170918m}
	{\rm Ker\,}\Box^{(q)}_{m}\cong H^{q}_{m}(\ol M).
	\end{equation}
	Hence,  ${\rm dim\,}H^q_m(\ol M)<+\infty$.
\end{theorem}


\section{The operators $\ddbar_\beta$, $\Box^{(q)}_\beta$ and reduction to the boundary}\label{s-gue171010}
In this section, we will show that if $m\gg1$, then $H^q_m(\ol M)\cong{\rm Ker\,}\hat\Box^{(q)}_{\beta,m}$, for every $q=0,1,\ldots,n-1$, where $\hat\Box^{(q)}_{\beta,m}$ is  an operator
analogous to the Kohn Laplacian defined on $X$(see \eqref{e-gue171014III}). From this result, we can reduce our problems to the boundary $X$. 
The methods of reduction to the boundary have a long history, going back to Calder\'on~\cite{Cal63}, Seeley~\cite{Seeley66}, H\"ormander~\cite{Hor66} and Boutet de Monvel~\cite{B71}.

\subsection{Poisson operator}
In this subsection, we introduce the Poisson operator on manifolds.
Fix $q\in\set{0,1,\ldots,n-1}$.
We first introduce some notations. We remind the reader that for $s\in\mathbb R$, the Sobolev space $W^{s}(\ol M,T^{*0,q}M')$ was introduced in the discussion after \eqref{e-gue171006yc}.
Let
\[\ddbar^\star_f: \Omega^{0,q+1}(M')\To\Omega^{0,q}(M')\]
be the formal adjoint of $\ddbar$ with respect to $(\,\cdot\,|\,\cdot\,)_{M'}$. That is
\[(\,\ddbar f\,|\,h\,)_{M'}=(\,f\,|\,\ddbar^\star_fh\,)_{M'},\]
$f\in\Omega^{0,q}_0(M')$, $h\in\Omega^{0,q+1}(M')$. Let
\[\Box^{(q)}_f=\ddbar\,\ddbar^\star_f+\ddbar^\star_f\,\ddbar: \Omega^{0,q}(M')\To\Omega^{0,q}(M')\]
denote the complex Laplace-Beltrami operator on $(0, q)$ forms. As before, let $\gamma$ denotes the operator of restriction to the boundary $X$. Let us consider the map
\begin{equation}\label{e-gue171010y}
\begin{split}
F^{(q)}:W^{2}(\ol M,T^{*0,q}M')&\rightarrow L^{2}_{(0,q)}(M)\oplus
W^{\frac{3}{2}}(X,T^{*0,q}M')\\
u&\mapsto (\Box_f^{(q)}u, \gamma u).
\end{split}
\end{equation}
From the elliptic regularity of $\Box_f^{(q) }$ with Dirichlet boundary condition, we know that
$\dim{\rm Ker\,}F^{(q)}<\infty$ and ${\rm Ker\,}F^{(q)}\subset \Omega^{0,q}(\ol M)$. Let
\begin{equation}\label{e-gue171011}
K^{(q)}: W^2(\overline M, T^{*0,q}M')\rightarrow{\rm Ker\,}F^{(q)}
\end{equation}
be the orthogonal projection with respect to $(\,\cdot\,|\,\cdot\,)_M$.  Put $\tilde \Box_f^{(q)}=\Box^{(q)}_f+K^{(q)}$ and consider the map
\begin{equation}\label{e-gue171010yI}
\begin{split}
\tilde F^{(q)}: W^2(\overline M, T^{*0,q}M')&\rightarrow L^{2}_{(0,q)}(M)\oplus W^{\frac{3}{2}}(X,T^{*0,q}M')\\
u&\mapsto (\tilde\Box_f^{(q)}u, \gamma u).
\end{split}
\end{equation}
Then $\tilde F^{(q)}$ is injective (see Chapter 3 in \cite{H08}). Let
\begin{equation}\label{e-gue171010yII}
\tilde P: C^\infty(X, T^{*0,q}M')\rightarrow\Omega^{0,q}(\overline M)
\end{equation}
be the Poisson operator for $\tilde \Box^{(q)}_f$ which is well-defined since \eqref{e-gue171010yI} is injective. The Poisson operator $\tilde P$ satisfies
\begin{equation}\label{e-gue171011II}
\begin{split}
&\tilde\Box^{(q)}_f\tilde Pu=0,\ \ \forall u\in C^\infty(X, T^{*0,q}M'),\\
&\gamma\tilde Pu=u,\ \  \forall u\in C^\infty(X, T^{*0,q}M').
\end{split}
\end{equation}

It is well-known that $\tilde P$ extends continuously
\[\tilde P: W^s(X, T^{*0,q}M')\rightarrow W^{s+\frac{1}{2}}(\overline M, T^{*0,q}M'),\ \ \forall s\in\mathbb R\]
(see page 29 of Boutet de Monvel~\cite{B71}). Let
$$\tilde P^\star: \hat{\mathcal E}'(\overline M, T^{*0,q}M')\rightarrow\mathcal D'(X, T^{*0,q}M')$$ be the operator defined by
\[(\,\tilde P^\star u\,|\,v\,)_X=(\,u\,|\,\tilde Pv\,)_M,\ \ u\in\hat{\mathcal E}'(\overline M, T^{*0,q}M'),\ \  v\in C^\infty(X, T^{*0,q}M'),\]
where $\hat{\mathcal E}'(\overline M, T^{*0,q}M')$ denotes the space of continuous linear maps from $\Omega^{0,q}(\ol M)$ to $\Complex$ with respect to $(\,\cdot\,|\,\cdot\,)_M$.
It is well-known (see page 30 of \cite{B71} ) that $\tilde P^\star$ is continuous: $\tilde P^\star: L^2_{(0,q)}(M)\rightarrow W^{\frac{1}{2}}(X, T^{*0,q}M')$ and
\[\tilde P^\star: \Omega^{0,q}(\overline M)\rightarrow C^\infty(X, T^{*0,q}M').\]

\begin{lemma}\label{comu0}
We have $T\circ K^{(q)}=K^{(q)}\circ T$ on $\Omega^{0, q}(\overline M)$.
\end{lemma}

\begin{proof}
Let $u\in\Omega^{0, q}(\overline M)$. We have $\Box^{(q)}_f K^{(q)}u=0$ and $\gamma K^{(q)}u=0$. Since $T$ commutes with $\Box^{(q)}_f$ and $\gamma$ we have $T\circ K^{(q)}u\in{\rm Ker\,}F^{(q)}$. Let $v:=K^{(q)}\circ Tu-T\circ K^{(q)}u$. Then, $v\in{\rm Ker\,}F^{(q)}$. We have
\begin{equation}
\begin{split}
(\,K^{(q)}\circ Tu-T\circ K^{(q)}u\,|\,v\,)_M&=(\,Tu\,|\,v\,)_M-(\,T\circ K^{(q)}u\,|\,v\,)_M\\
&=(\,Tu\,|\,v\,)_M+(\,K^{(q)}u\,|\,Tv\,)_M\\
&=(\,Tu\,|\,v\,)_M+(\,u\,|\,Tv\,)_M=0.
\end{split}
\end{equation}
Thus, $K^{(q)}\circ Tu-T\circ K^{(q)}u=0$. The lemma follows.
\end{proof}

We also need

\begin{lemma}\label{comu0d}
	We have $T\circ\tilde P=\tilde P\circ T$ on $C^\infty(X,T^{*0,q}M')$.
\end{lemma}

\begin{proof}
	Let $u\in C^\infty(X,T^{*0,q}M')$.
	Since $T\Box^{(q)}_{f}=\Box^{(q)}_{f}T$, $TK^{(q)}=K^{(q)}T$ and $T\gamma=\gamma T$, we have
	$\tilde F^{(q)}(T\tilde Pu-\tilde PTu)=0$. Since $\tilde F^{(q)}$ is injective, we get $T\circ\Td Pu=\Td P\circ Tu$.
\end{proof}

From Lemma~\ref{comu0} and Lemma~\ref{comu0d}, we have
\begin{equation}\label{e-gue171030a}
K^{(q)}: \Omega^{0,q}_m(\ol M)\To\Omega^{0,q}_m(\ol M),  \tilde P: \Omega^{0,q}_m(X)\To\Omega^{0,q}_m(\ol M),\ \ \forall m\in\mathbb Z.
\end{equation}

For every $m\in\mathbb Z$, let $Q^{(q)}_m: L^2_{(0,q)}(M)\To L^2_{(0,q),m}(M)$ be the orthogonal projection with respect to $(\,\cdot\,|\,\cdot\,)_M$. Let $u\in{\rm Ker\,}F^{(q)}$. It is easy to see that $Q^{(q)}_mu\in{\rm Ker\,}F^{(q)}$. For every $m\in\mathbb Z$, put ${\rm Ker\,}F^{(q)}_m:={\rm Ker\,}F^{(q)}\bigcap\Omega^{0,q}_m(\ol M)$. It is easy to see that ${\rm Ker\,}F^{(q)}_m\perp {\rm Ker\,}F^{(q)}_{m'}$, for every $m, m'\in\mathbb Z$, $m\neq m'$, and ${\rm Ker\,}F^{(q)}=\oplus_{m\in\mathbb Z}{\rm Ker\,}F^{(q)}_m$. From ${\rm dim\,}{\rm Ker\,}F^{(q)}<+\infty$ and \eqref{e-gue171030a}, we deduce that

\begin{lemma}\label{t-gue171011}
There is a $m_0\in\mathbb N$, such that for every $m\in\mathbb Z$ with $\abs{m}\geq m_0$, we have ${\rm Ker\,}F^{(q)}_m=\set{0}$ and $K^{(q)}u=0$, for every $u\in\Omega^{0,q}_m(\ol M)$, where $K^{(q)}$ is as in \eqref{e-gue171011}.
\end{lemma}

It should be mentioned that in this section, before Lemma~\ref{t-gue171011}, we do not need the hypothesis of largeness of the frequencies.


From Lemma~\ref{t-gue171011}, \eqref{e-gue171030a} and \eqref{e-gue171011II}, we deduce

\begin{lemma}\label{t-gue171011I}
There is a $m_0\in\mathbb N$, such that for every $m\in\mathbb Z$ with $\abs{m}\geq m_0$, we have
\begin{equation}\label{e-gue171011IId}
\begin{split}
&\Box^{(q)}_f\Td Pu=0,\ \ \forall u\in\Omega^{0,q}_m(X),\\
&\gamma\Td Pu=u,\ \  \forall u\in\Omega^{0,q}_m(X).
\end{split}
\end{equation}
\end{lemma}

Only when $|m|$ is sufficiently large we can have that the Poisson operator is a Harmonic extension operator from the boundary to the open complex manifold. From now on, we assume that $\abs{m}\geq m_0$, where $m_0\in\mathbb N$ is as in Lemma~\ref{t-gue171011I}. 
\subsection{$\overline\partial_\beta$-operator }
We define a new inner product on $W^{-\frac{1}{2}}(X,T^{*0,q}M')$ as follows:
\begin{equation}\label{inner product}
[\,u\,|\,v\,]=(\,\tilde Pu\,|\,\tilde Pv)_{M}.
\end{equation}

Let
\begin{equation}\label{e-gue171011pm}
Q: W^{-\frac{1}{2}}(X,T^{*0,q}M')\rightarrow{\rm Ker\,}\ddbar\rho^{\wedge,\star}=W^{-\frac{1}{2}}(X,T^{*0,q}X)
\end{equation}
be the orthogonal projection onto ${\rm Ker\,}\ddbar\rho^{\wedge,\star}$ with respect to $[\,\cdot\,|\,\cdot\,]$.

The following is well-known (see Lemma 4.2 of the second part in~\cite{H08} )

\begin{lemma} \label{l-gue171013}
$Q$ is a classical pseudodifferential operator of order $0$ with principal symbol
$2(\ddbar\rho)^{\wedge, \star}(\ddbar\rho)^\wedge$.
Moreover,
\begin{equation} \label{e-gue171013x}
I-Q=(\tilde P^\star\tilde P)^{-1}(\ddbar\rho)^\wedge S,
\end{equation}
where $S:C^\infty(X, T^{*0,q}M')\To C^\infty(X,T^{*0,q}M')$
is a classical pseudodifferential operator of order $-1$.
\end{lemma}

\begin{remark}\label{r-gue171013}
It is well-known that the operator
\[\tilde P^\star\tilde P:C^\infty(X,T^{*0,q}M')\To C^\infty(X,T^{*0,q}M')\]
is a classical elliptic pseudodifferential operator of order $-1$ and invertible since $\tilde P$ is
injective~(see Boutet de Monvel~\cite{B71}). Moreover, the operator
\[(\tilde P^\star\tilde P)^{-1}:C^\infty(X,T^{*0,q}M')\To C^\infty(X,T^{*0,q}M')\]
is a classical elliptic pseudodifferential operator of order $1$.
\end{remark}

We consider the following operator
\begin{equation}\label{e-gue171011pmI}
\ddbar_{\beta}=Q\gamma\ddbar \tilde P:\Omega^{0,q}(X)\rightarrow\Omega^{0,q+1}(X).
\end{equation}
The operator $\ddbar_\beta$ was introduced by the first author in~\cite{H08}. It is straightforward to see that
\begin{equation}\label{e-gue171030aII}
\mbox{$T\circ Q=Q\circ T$ on $\Omega^{0,q}(X)$}.
\end{equation}
From \eqref{e-gue171030aII} and
Lemma~\ref{comu0d}, we deduce that
\begin{equation}\label{e-gue171012b}
T\circ \ddbar_{\beta}= \ddbar_{\beta}\circ T\ \ \mbox{on $\Omega^{0,q}(X)$}.
\end{equation}
Let $\ddbar^\dagger_\beta:\Omega^{0,q+1}(X)\To\Omega^{0,q}(X)$ be the formal adjoint with respect to $[\,\cdot\,|\,\cdot\,]$. It is not difficult to see that
\begin{equation}\label{e-gue171012bI}
T\circ \ddbar^\dagger_{\beta}= \ddbar^\dagger_{\beta}\circ T\ \ \mbox{on $\Omega^{0,q}(X)$}.
\end{equation}
Set
\begin{equation}\label{e-gue171012bII}
\Box^{(q)}_{\beta}=\ddbar_{\beta}^{\dagger}\ddbar_{\beta}+\ddbar_{\beta}\ddbar_{\beta}^{\dagger}:\Omega^{0,q}(X)\To\Omega^{0,q}(X).
\end{equation}
From \eqref{e-gue171012b} and \eqref{e-gue171012bI}, we deduce that
\begin{equation}\label{e-gue171012bIII}
T\circ\Box^{(q)}_{\beta}= \Box^{(q)}_{\beta}\circ T\ \ \mbox{on $\Omega^{0,q}(X)$}.
\end{equation}
We write $\Box^{(q)}_{\beta,m}$ to denote the restriction of $\Box^{(q)}_{\beta}$ on $\Omega^{0,q}_m(X)$.

We pause and recall the tangential Cauchy-Riemann operator and Kohn Laplacian.
For $x\in X$, let
$\pi^{0,q}_x:T^{*0,q}_xM'\To T^{*0,q}_xX$
be the orthogonal projection map with respect to $\langle\,\cdot\,|\,\cdot\,\rangle$. The tangential Cauchy-Riemann operator is given by
\[\ddbar_b:=\pi^{0,q+1}\circ d:\Omega^{0,q}(X)\To\Omega^{0,q+1}(X).\]
Let $\ddbar^\star_b$ be the formal adjoint of $\ddbar_b$ with respect to $(\,\cdot\,|\,\cdot\,)_X$, that is
$(\,\ddbar_{b}f\,|\,\,h)_X=(f\,|\,\ddbar^\star_bh\,)_X$, $f\in\Omega^{0,q}(X)$, $h\in\Omega^{0,q+1}(X)$. The Kohn Laplacian is given by
\begin{equation}\label{e-gue171011pmy}
\Box^{(q)}_b:=\ddbar_b\,\ddbar^\star_b+\ddbar^\star_b\,\ddbar_b:\Omega^{0,q}(X)\To\Omega^{0,q}(X).
\end{equation}
It is well-known that (see~\cite{Hsiao14} and~\cite{HL15})
\begin{equation}\label{e-gue171012}
\begin{split}
T\circ\ddbar_b=\ddbar_b\circ T\ \ \mbox{on $\Omega^{0,q}(X)$},\\
T\circ\ddbar^\star_b=\ddbar^\star_b\circ T\ \ \mbox{on $\Omega^{0,q}(X)$},\\
T\circ\Box^{(q)}_b=\Box^{(q)}_b\circ T\ \ \mbox{on $\Omega^{0,q}(X)$}.
\end{split}
\end{equation}
We write $\Box^{(q)}_{b,m}$ to denote the restriction of $\Box^{(q)}_{b}$ to $\Omega^{0,q}_m(X)$.
We come back to our situation. We recall the following (see Lemma 5.1, equation (5.3) and Lemma 5.2 in the second part of~\cite{H08})

\begin{proposition}\label{t-gue171012}
We have that $\ddbar_\beta$ and $\ddbar^\dagger_\beta$ are classical pseudodifferential operators of order $1$,
\begin{equation}\label{e-gue171012a}
\ddbar_{\beta}\circ \ddbar_{\beta}=0\ \ \mbox{on $\Omega^{0,q}(X)$}
\end{equation}
and
\begin{equation}\label{e-gue171012aI}
\begin{split}
&\mbox{$\ddbar_{\beta}=\ddbar_{b}$+lower order terms},\\
&\mbox{$\ddbar_{\beta}^{\dagger}=\gamma\ddbar^{\star}_{f}\tilde P=\ddbar^{\star}_{b}$+lower order terms}.
\end{split}
\end{equation}
\end{proposition}

\subsection{Reduction to boundary}
We introduce some notations. Let
\[\triangle_X:=-\frac{1}{2}(dd^\star+d^\star d):C^\infty(X,\Lambda^q(\Complex T^*X))\To C^\infty(X,\Lambda^q(\Complex T^*X))\]
be the real Laplacian on $X$, where $d: C^\infty(X,\Lambda^q(\Complex T^*X))\To C^\infty(X,\Lambda^{q+1}(\Complex T^*X))$ is the standard derivative
and $d^\star: C^\infty(X,\Lambda^{q+1}(\Complex T^*X))\To C^\infty(X,\Lambda^{q}(\Complex T^*X))$ is the formal adjoint of $d$ with respect to $(\,\cdot\,|\,\cdot\,)_X$. We extend $-\triangle_X$ to $L^2$ space:
\[-\triangle_X: {\rm Dom\,}(-\triangle_X)\subset L^2(X,\Lambda^q(\Complex T^*X))\To L^2(X,\Lambda^q(\Complex T^*X)),\]
where ${\rm Dom\,}(-\triangle_X)=\set{u\in L^2(X,\Lambda^q(\Complex T^*X));\, -\triangle_Xu\in L^2(X,\Lambda^q(\Complex T^*X))}$.
Then, $-\triangle_X$  is a non-negative operator. Let $\sqrt{-\triangle_X}$ be the square root of $-\triangle_X$.
Then, $\sqrt{-\triangle_X}$ is a non-negative operator, has $L^2$ closed range, ${\rm Ker\,}\sqrt{-\triangle_X}$ is a finite dimensional subspace of $C^\infty(X,\Lambda^q(\Complex T^*X))$. Moreover, it is easy to see that
\begin{equation}\label{e-gue171012r}
T\circ\sqrt{-\triangle_X}=\sqrt{-\triangle_X}\circ T\ \ \mbox{on $C^\infty(X,\Lambda^q(\Complex T^*X))$}.
\end{equation}
Let $G: L^2(X,\Lambda^q(\Complex T^*X))\To{\rm Dom\,}\sqrt{-\triangle_X}$ be the partial inverse of $\sqrt{-\triangle_X}$ and let $H: L^2(X,\Lambda^q(\Complex T^*X))\To{\rm Ker\,}\sqrt{-\triangle_X}$ be the orthogonal projection. We have
\begin{equation}\label{e-gue171012rb}
\begin{split}
&\sqrt{-\triangle_X}G+H=I\ \ \mbox{on $L^2(X,\Lambda^q(\Complex T^*X))$},\\
&G\sqrt{-\triangle_X}+H=I\ \ \mbox{on ${\rm Dom\,}\sqrt{-\triangle_X}$}.
\end{split}
\end{equation}
Note that $G$ is a classical pseudodifferentail operator of order $-1$ and $H$ is a smoothing operator.


For $x\in X$, put
\begin{equation} \label{e-gue171012t}
I^{0,q}T^*_xM'=\set{u\in T^{*0,q}_xM';\, u=(\ddbar\rho)^\wedge g,\ g\in T^{*0,q-1}_xM'}
\end{equation}
and let $I^{0,q}T^*M'$ be the vector bundle over $X$ with fiber $I^{0,q}T^*_xM'$, $x\in X$. Note that for every $x\in M'$,
$I^{0,q}T^*_x(M)$ is orthogonal to $T^{*0,q}_xX$.
We recall the following (see Proposition 4.2 of Part II in~\cite{H08})

\begin{proposition} \label{t-gue171012t}
The operator
\[\gamma(\ddbar\rho)^{\wedge}\ddbar^\star_f\Td P(\Td P^\star\Td P)^{-1}: C^\infty(X, I^{0,q}T^*M')\To C^\infty(X,I^{0,q}T^*M')\]
is a classical pseudodifferential operator of order two,
\begin{equation} \label{e-gue171012tII}
\gamma(\ddbar\rho)^{\wedge}\ddbar^\star_f\Td P(\Td P^\star\Td P)^{-1}=(iT-\sqrt{-\triangle_X})\sqrt{-\triangle_X}+R,
\end{equation}
where the reminder term $R$ has order $\leq 1$.
\end{proposition}

For every $m\in\mathbb Z$, put $C^\infty_m(X, I^{0,q}T^*M'):=\set{u\in C^\infty(X, I^{0,q}T^*M');\, Tu=imu}$. Note that $I^{0,q}T^*M'$ is a $S^1$-invariant vector bundle over $X$. It is not difficult to see that
\begin{equation}\label{e-gue171012u}
T\circ \gamma(\ddbar\rho)^{\wedge}\ddbar^\star_f\Td P(\Td P^\star\Td P)^{-1}=\gamma(\ddbar\rho)^{\wedge}\ddbar^\star_f\Td P(\Td P^\star\Td P)^{-1}\circ T\ \ \mbox{on $C^\infty(X, I^{0,q}T^*M')$}
\end{equation}
and hence
\begin{equation}\label{e-gue171012uI}
\gamma(\ddbar\rho)^{\wedge}\ddbar^\star_f\Td P(\Td P^\star\Td P)^{-1}: C^\infty_m(X, I^{0,q}T^*M')\To C^\infty_m(X, I^{0,q}T^*M'),\ \ \forall m\in\mathbb Z.
\end{equation}

We can now prove

\begin{lemma} \label{t-gue171012ta}
There is a $\hat m_0\in\mathbb N$ such that for every $m\geq\hat m_0$, $m\in\mathbb N$, the operator
\[\gamma(\ddbar\rho)^{\wedge}\ddbar^\star_f\Td P(\Td P^\star\Td P)^{-1}: C^\infty_m(X, I^{0,q}T^*M')\To C^\infty_m(X,I^{0,q}T^*M')\]
is injective.
\end{lemma}

Note that Lemma \ref{t-gue171012ta} is the key point why $m$ is positive in \eqref{e-gue171026I}.
\begin{proof}
Since ${\rm Ker\,}\sqrt{-\triangle_X}$ is a finite dimensional subspace of $C^\infty(X,\Lambda^q(\Complex T^*M'))$, there is a $m_1\in\mathbb N$, such that for every $m\geq m_1$, $m\in\mathbb N$, we have
\begin{equation}\label{e-gue171012e}
{\rm Ker\,}\sqrt{-\triangle_X}\bigcap C^\infty_m(X,\Lambda^q(\Complex T^*M'))=\set{0},
\end{equation}
where $C^\infty_m(X,\Lambda^q(\Complex T^*M'))=\set{u\in C^\infty(X,\Lambda^q(\Complex T^*M'));\, Tu=imu}$. Let $G, H$ be as in \eqref{e-gue171012rb}. From \eqref{e-gue171012r}, it is not difficult to see that
\[H: C^\infty_m(X,\Lambda^q(\Complex T^*M'))\To C^\infty_m(X,\Lambda^q(\Complex T^*M')),\ \ \forall m\in\mathbb Z.\]
From this observation and \eqref{e-gue171012e}, we conclude that
\begin{equation}\label{e-gue171012eI}
{\rm Ker\,}\sqrt{-\triangle_X}\perp C^\infty_m(X,\Lambda^q(\Complex T^*M')),\ \ \forall m\in\mathbb Z,\ \ m\geq m_1.
\end{equation}

Let $u\in C^\infty_m(X,I^{0,q}T^*M')$, $m\geq m_1$. From \eqref{e-gue171012tII}, \eqref{e-gue171012eI} and \eqref{e-gue171012rb}, we have
\begin{equation}\label{e-gue171012eIII}
\begin{split}
&\gamma(\ddbar\rho)^{\wedge}\ddbar^\star_f\Td P(\Td P^\star\Td P)^{-1}u\\
&=(iT-\sqrt{-\triangle_X})(\sqrt{-\triangle_X}u)+R\circ G\circ(\sqrt{-\triangle_X}u).
\end{split}
\end{equation}
Since $R\circ G$ is a classical pseudodifferential operator of order $0$, there is a constant $C>0$ independent of $m$ and $u$ such that
\begin{equation}\label{e-gue171012s}
\norm{R\circ G\circ(\sqrt{-\triangle_X}u)}_X\leq C\norm{\sqrt{-\triangle_X}u}_X.
\end{equation}
Since $T\sqrt{-\triangle_X}u=im\sqrt{-\triangle_X}u$ and $\sqrt{-\triangle_X}$ is a non-negative operator, we deduce that
\begin{equation}\label{e-gue171012sI}
\norm{(iT-\sqrt{-\triangle_X})(\sqrt{-\triangle_X}u)}_X\geq m\norm{\sqrt{-\triangle_X}u}_X.
\end{equation}
From \eqref{e-gue171012eIII}, \eqref{e-gue171012s}, \eqref{e-gue171012sI} and \eqref{e-gue171012eI} we conclude that there is a $m_2\in\mathbb N$ with $m_2>m_1\in\mathbb N$, such that for every $u\in C^\infty_m(X,I^{0,q}T^*M')$, $m\geq m_2$, if
\[\gamma(\ddbar\rho)^{\wedge}\ddbar^\star_f\Td P(\Td P^\star\Td P)^{-1}u=0,\]
then $\sqrt{-\triangle_X}u=0$ and hence $u=0$. The theorem follows.
\end{proof}

Put ${\rm Ker\,}\Box^{(q)}_{\beta,m}:=\set{u\in\Omega^{0,q}_m(X);\, \Box^{(q)}_{\beta,m}u=0}$. As Kohn Laplacian case, $\Box^{(q)}_{\beta,m}$ is a transversally elliptic operator and we have ${\rm dim\,}{\rm Ker\,}\Box^{(q)}_{\beta,m}<+\infty$ (see Section 3 in~\cite{CHT}) .
We can now prove the main result of this section

\begin{theorem}\label{t-gue171013}
There is a $\tilde m_0\in\mathbb N$, such that for all $m\geq\tilde m_0$, $m\in\mathbb N$, we have
\[H^{q}_{m}(\ol M)\cong{\rm Ker\,}\Box^{(q)}_{\beta,m}.\]
\end{theorem}

\begin{proof}
We assume that $m\geq m_0$, where $m_0\in\mathbb N$ is as in Lemma~\ref{t-gue171011I}. Thus, \eqref{e-gue171011IId} hold.
Put
\[W=\{u\in\Omega^{0,q}_{m}(\ol M);\,
\ddbar u=0, \ddbar^{\star}_{f}u=0, \gamma\ddbar\rho^{\wedge,\star}u=0,
\gamma\ddbar\rho^{\wedge,\star}\ddbar u=0  \}.\]
By \eqref{Neum1} and the isomorphism \eqref{iso170918m}, we only need to show that ${\rm Ker\,}\Box^{(q)}_{\beta, m}\cong W$.
We consider the map
\begin{equation}\label{e-gue171013a}
\begin{split}
F: W&\rightarrow{\rm Ker\,}\Box^{(q)}_{\beta, m},\\
u&\mapsto v=\gamma u. \\
\end{split}
\end{equation}
We now show that the map $F$ is well-defined. Let $u\in W$. Since $\gamma\ddbar\rho^{\wedge,\star}u=0$, we have
$v:=\gamma u\in C^{\infty}_{m}(X,T^{*0,q}X)$. Since $\Box^{(q)}_{f}u=0$, we have $u=\tilde P\gamma u$ and hence
$\ddbar_{\beta}v=Q\gamma\ddbar \tilde P\gamma u=Q\gamma\ddbar u=0$. From \eqref{e-gue171012aI}, we have
\[\ddbar^\dagger_\beta v=\gamma\ddbar^\star_f\tilde Pv=\gamma\ddbar^\star_fu=0.\]
We have proved that $v:=\gamma u\in{\rm Ker\,}\Box^{(q)}_\beta$ and the map $F$ is well-defined.

Let $u\in W$. If $v:=F(u)=\gamma u=0$, then $u=\tilde Pv=0$. Hence the map is injective. Now, we prove that the map $F$ is surjective. Fix $v\in{\rm Ker\,}\Box^{(q)}_{\beta, m}$. We are going to prove that $\tilde Pv\in W$ if $m\geq\sup\set{m_0, \hat m_0}$, where $\hat m_0\in\mathbb N$ is as in Lemma~\ref{t-gue171012ta}.
Since $v\in\Omega^{0,q}_{m}(X)$, we have
\begin{equation}\label{e-gue171013pm}
\gamma\ddbar\rho^{\wedge,\star}\tilde Pv=\ddbar\rho^{\wedge,\star}v=0.
\end{equation}
Since $\Box^{(q)}_{\beta,m}v=0$, we have
\begin{equation}\label{e-gue171013p}
\ddbar_{\beta}v=Q\gamma\ddbar \tilde Pv=0
\end{equation}
and
\begin{equation}\label{e-gue171013pI}
\ddbar^{\dagger}_{\beta}v=\gamma\ddbar^{\star}_{f}\tilde Pv=0.
\end{equation}
Combining \eqref{e-gue171013p}, \eqref{e-gue171013pI} with $\gamma(\ddbar\,\ddbar^\star_f+\ddbar^\star_f\,\ddbar)\tilde Pv=0$, we have
\[\gamma\ddbar^\star_f\tilde P\gamma\ddbar\tilde Pv=-\gamma\ddbar\tilde P\gamma\ddbar^\star_f\tilde Pv=0\]
and
\begin{equation} \label{e-gue171013pII}
\gamma\ddbar^\star_f\tilde P(I-Q)\gamma\ddbar\tilde Pv=\gamma\ddbar^\star_f\tilde P\gamma\ddbar\tilde Pv
-\gamma\ddbar^\star_f\tilde PQ\gamma\ddbar\tilde Pv=0.
\end{equation}
Combining \eqref{e-gue171013pII} with \eqref{e-gue171013x}, we get
\[\gamma\ddbar^\star_f\tilde P(I-Q)\gamma\ddbar\tilde Pv=\gamma\ddbar^\star_f\tilde P(\tilde P^\star\tilde P)^{-1}(\ddbar\rho)^\wedge S\gamma\ddbar \tilde Pv= 0.\]
Thus,
\begin{equation}\label{e-gue171013pIII}
\gamma(\ddbar\rho)^\wedge\ddbar^\star_f\tilde P(\tilde P^\star\tilde P)^{-1}(\ddbar\rho)^\wedge S\gamma\ddbar \tilde Pv= 0.
\end{equation}
In view of Lemma~\ref{t-gue171012ta}, we see that for every $m\geq\hat m_0$, $m\in\mathbb N$, the operator
\[\gamma(\ddbar\rho)^{\wedge}\ddbar^\star_f\Td P(\Td P^\star\Td P)^{-1}: C^\infty_m(X, I^{0,q}T^*M')\To C^\infty_m(X,I^{0,q}T^*M')\]
is injective. From this observation, \eqref{e-gue171013x} and \eqref{e-gue171013pIII}, we conclude that if $m\geq\hat m_0$, we have
$(\ddbar\rho)^\wedge S\gamma\ddbar \tilde Pv=0$ and
\begin{equation}\label{e-gue171013mI}
(I-Q)\gamma\ddbar \tilde Pv=(\tilde P^\star\tilde P)^{-1}(\ddbar\rho)^\wedge S\gamma\ddbar \tilde Pv=0.
\end{equation}
From \eqref{e-gue171013p} and \eqref{e-gue171013mI}, we deduce that
\begin{equation}\label{e-gue171013mII}
\gamma\ddbar\tilde Pv=0.
\end{equation}
From \eqref{e-gue171013pm}, \eqref{e-gue171013pI} and \eqref{e-gue171013mII}, we deduce that $\tilde Pv\in W$ and hence $F$ is surjective. The theorem follows.
\end{proof}

In view of Theorem~\ref{t-gue171013}, the study of the space $H^q_m(\ol M)$ is equivalent to the study of the space ${\rm Ker\,}\Box^{(q)}_{\beta,m}$.
The operator $\Box^{(q)}_{\beta,m}$ is similar to Kohn Laplacian $\Box^{(q)}_{b,m}$. The big difference is that $\Box^{(q)}_{\beta,m}$ is self-adjoint with respect to $[\,\cdot\,|\,\cdot\,]$, but $\Box^{(q)}_{b,m}$ is self-adjoint with respect to $(\,\cdot\,|\,\cdot\,)_X$. In order to apply our analysis about $\Box^{(q)}_{b,m}$ (see~\cite{HL15}), we will introduce another operator $\hat\Box^{(q)}_{\beta,m}$ which is similar to $\Box^{(q)}_{\beta,m}$ but self-adjoint with respect to $(\,\cdot\,|\,\cdot\,)_X$.

Let $\tilde Q: L^2(X, T^{*0,q}M')\To L^2_{(0,q)}(X)$ be the orthogonal projection with respect to $(\,\cdot\,|\,\cdot\,)_X$. In view of Remark~\ref{r-gue171013}, we see that the operator $\tilde Q\tilde P^\star\tilde P:\Omega^{0,q}(X)\To\Omega^{0,q}(X)$
is a classical elliptic pseudodifferential operator of order $-1$ and invertible and the operator
$(\tilde Q\tilde P^\star\tilde P)^{-1}:\Omega^{0,q}(X)\To\Omega^{0,q}(X)$
is a classical elliptic pseudodifferential operator of order $1$. Let $(\tilde Q\tilde P^\star\tilde P)^{\frac{1}{2}} :\Omega^{0,q}(X)\To\Omega^{0,q}(X)$
be the square root of $\tilde Q\tilde P^\star\tilde P$ and let
$(\tilde Q\tilde P^\star\tilde P)^{-\frac{1}{2}} : \Omega^{0,q}(X)\To\Omega^{0,q}(X)$
be the square root of $(\tilde Q\tilde P^\star\tilde P)^{-1}$. Then, $(\tilde Q\tilde P^\star\tilde P)^{\frac{1}{2}}$ is a classical elliptic pseudodifferential operator of order $-\frac{1}{2} $ and $(\tilde Q\tilde P^\star\tilde P)^{-\frac{1}{2}}$ is a classical elliptic pseudodifferential operator of order $\frac{1}{2}$. Let
\begin{equation}\label{e-gue171014}
\hat\ddbar_{\beta}:=(\tilde Q\tilde P^{\star}\tilde P)^{\frac{1}{2}}\ddbar_{\beta}(\tilde Q\tilde P^{\star}\tilde P)^{-\frac{1}{2}}: \Omega^{0, q}(X)\rightarrow\Omega^{0,q+1}(X)
\end{equation}
and let $\hat\ddbar_{\beta}^{\star}: \Omega^{0,q+1}(X)\To\Omega^{0,q}(X)$ be the formal adjoint of $\hat\ddbar_\beta$ with respect to $(\,\cdot\,|\,\cdot\,)_X$. By direct calculation, we can check that
\begin{equation}\label{e-gue171014I}
\hat\ddbar_{\beta}^{\star}=(\tilde Q\tilde P^{\star}\tilde P)^{\frac{1}{2}}\ddbar_{\beta}^{\dagger}(\tilde Q\tilde P^{\star}\tilde P)^{-\frac{1}{2}}.
\end{equation}
Moreover, it is easy to see that
\begin{equation}\label{e-gue171014II}
\begin{split}
&T\circ\hat\ddbar_{\beta}=\hat\ddbar_{\beta}\circ T\ \ \mbox{on $\Omega^{0,q}(X)$},\\
&T\circ\hat\ddbar^\star_{\beta}=\hat\ddbar^\star_{\beta}\circ T\ \ \mbox{on $\Omega^{0,q}(X)$}.
\end{split}
\end{equation}
Set
\begin{equation}\label{e-gue171014III}
\hat\Box^{(q)}_{\beta}=\hat\ddbar_{\beta}^{\star}\,\hat\ddbar_{\beta}+\hat\ddbar_{\beta}\,\hat\ddbar_{\beta}^{\star}:\Omega^{0,q}(X)\To\Omega^{0,q}(X).
\end{equation}
From \eqref{e-gue171014II}, we have
\begin{equation}\label{e-gue171014a}
T\circ\hat\Box^{(q)}_{\beta}=\hat\Box^{(q)}_{\beta}\circ T\ \ \mbox{on $\Omega^{0,q}(X)$}.
\end{equation}
We write $\hat\Box^{(q)}_{\beta,m}$ to denote the restriction of $\hat\Box^{(q)}_{\beta}$ to $\Omega^{0,q}_m(X)$.

From Proposition~\ref{t-gue171012}, \eqref{e-gue171014} and \eqref{e-gue171014I}, we deduce that

\begin{proposition}\label{t-gue171014}
We have $\hat\ddbar_\beta$ and $\hat\ddbar^\star_\beta$ are classical pseudodifferential operators of order $1$,
\begin{equation}\label{e-gue171014aI}
\hat\ddbar_{\beta}\circ\hat\ddbar_{\beta}=0\ \ \mbox{on $\Omega^{0,q}(X)$}
\end{equation}
and
\begin{equation}\label{e-gue171014aII}
\begin{split}
&\mbox{$\hat\ddbar_{\beta}=\ddbar_{b}$+lower order terms},\\
&\mbox{$\hat\ddbar_{\beta}^{\star}=\ddbar^{\star}_{b}$+lower order terms}.
\end{split}
\end{equation}
\end{proposition}

Fix $m\in\mathbb Z$, put ${\rm Ker\,}\hat\Box^{(q)}_{\beta,m}:=\set{u\in\Omega^{0,q}_m(X);\, \hat\Box^{(q)}_{\beta,m}u=0}$. It is clear that the map
\begin{equation}\label{e-gue171014t}
\begin{split}
\psi: {\rm Ker\,}\Box^{(q)}_{\beta,m}&\rightarrow{\rm \Ker\,}\hat\Box^{(q)}_{\beta,m},\\
u&\mapsto(\tilde Q\tilde P^{\star}\tilde P)^{\frac{1}{2}}u
\end{split}
\end{equation}
is an isomorphism. Hence, ${\rm Ker\,}\hat\Box^{(q)}_{\beta,m}\cong{\rm Ker\,}\Box^{(q)}_{\beta,m}$, for every $m\in\mathbb Z$. From this observation and Theorem~\ref{t-gue171013}, we deduce that

\begin{theorem}\label{t-gue171014pm}
There is a $\tilde m_0\in\mathbb N$, such that for all $m\geq\tilde m_0$, $m\in\mathbb N$, we have
\[H^{q}_{m}(\ol M)\cong{\rm Ker\,}\hat\Box^{(q)}_{\beta,m}.\]
\end{theorem}

\subsection{Technical lemmas for $\hat\Box^{(q)}_\beta$}\label{s-gue200131yyd}
In the last part of this section, we state several technial lemmas for $\hat\Box^{(q)}_\beta$. These lemmas will be needed in the proof of Morse inequalities in Section \ref{s-gue171021}.

We can apply Kohn's $L^2$ estimates to $\hat\Box^{(q)}_\beta$ (see~\cite[Theorem 8.4.2]{CS01}) and get
\begin{proposition}\label{t-gue171016}
For every $s\in\mathbb N_0$, there is a constant $C_s>0$ such that
\begin{equation}\label{e-gue171016j}
\norm{u}_{s+1,X}\leq C_s\Bigr(\norm{\hat\Box^{(q)}_\beta u}_{s,X}+\norm{Tu}_{s,X}+\norm{u}_X\Bigr),\ \ \forall u\in\Omega^{0,q}(X),
\end{equation}
where $\norm{\cdot}_{s,X}$ denotes the usual Sobolev norm of order $s$ on $X$.
\end{proposition}

We need

\begin{lemma}\label{l-gue171016}
Let $P: \Omega^{0,q}(X)\To\Omega^{0,q}(X)$ be a classical pseudodifferential operator of order one. For every $k\in\mathbb N$, we have
\begin{equation}\label{e-gue171016}
(\hat\Box^{(q)}_\beta)^k\circ P=\sum^{k+1}_{j=1}T_j\circ(\hat\Box^{(q)}_\beta)^{k+1-j},
\end{equation}
where $T_j: \Omega^{0,q}(X)\To\Omega^{0,q}(X)$ is a classical pseudodifferential operator of order $j$, $j=1,2,\ldots,k+1$.
\end{lemma}

\begin{proof}
We have $\hat\Box^{(q)}_\beta\circ P=P\circ\hat\Box^{(q)}_\beta+[P, \hat\Box^{(q)}_\beta]$. Since the principal symbol of $\hat\Box^{(q)}_\beta$ is scalar,
$[P, \hat\Box^{(q)}_\beta] :\Omega^{0,q}(X)\To\Omega^{0,q}(X)$ is a pseudodifferential operator of order $2$. Thus, \eqref{e-gue171016} holds for $k=1$. Assume that
\eqref{e-gue171016} holds for $k_0\in\mathbb N$. We want to show that \eqref{e-gue171016} holds for $k_0+1$. By the induction assumption, we have
\begin{equation}\label{e-gue171016r}
(\hat\Box^{(q)}_\beta)^{k_0}\circ P=\sum^{k_0+1}_{j=1}T_j\circ(\hat\Box^{(q)}_\beta)^{k_0+1-j},
\end{equation}
where $T_j: \Omega^{0,q}(X)\To\Omega^{0,q}(X)$ is a classical pseudodifferential operator of order $j$. From \eqref{e-gue171016r}, we have
\begin{equation}\label{e-gue171016rI}
\begin{split}
(\hat\Box^{(q)}_\beta)^{k_0+1}\circ P&=(\hat\Box^{(q)}_\beta)\circ\Bigr(\sum^{k_0+1}_{j=1}T_j\circ(\hat\Box^{(q)}_\beta)^{k+1-j})\\
&=\sum^{k_0+1}_{j=1}T_j\circ(\hat\Box^{(q)}_\beta)^{k_0+2-j}+\sum^{k_0+1}_{j=1}[\hat\Box^{(q)}_\beta, T_j]\circ (\hat\Box^{(q)}_\beta)^{k_0+1-j}\\
&=T_1\circ(\hat\Box^{(q)}_\beta)^{k_0+1}+\sum^{k_0+1}_{j=2}\Bigr(T_j+[\hat\Box^{(q)}_\beta, T_{j-1}]\Bigr)\circ(\hat\Box^{(q)}_\beta)^{k_0+2-j}+[\hat\Box^{(q)}_\beta, T_{k_0+1}].
\end{split}
\end{equation}
Note that $T_j+[\hat\Box^{(q)}_\beta, T_{j-1}]: \Omega^{0,q}(X)\To\Omega^{0,q}(X)$ is a classical pseudodifferential operator of order $j$. From this observation and \eqref{e-gue171016rI}, we get \eqref{e-gue171016} for $k=k_0+1$. By the induction assumption, \eqref{e-gue171016} follows.
\end{proof}

We can now prove

\begin{lemma}\label{t-gue171019}
Fix $m\in\mathbb Z$. For every $k\in\mathbb N$, there is a constant $C_k>0$ independent of $m$ such that
\begin{equation}\label{e-gue171019}
\norm{(\frac{1}{m}\Box^{(q)}_b)^ku}_X\leq C_k\Bigr(\sum^{k}_{\ell=0}\norm{(\frac{1}{m}\hat\Box^{(q)}_\beta)^\ell u}_X\Bigr),\ \ \forall u\in\Omega^{0,q}_m(X).
\end{equation}
\end{lemma}

\begin{proof}
We have
\begin{equation}\label{e-gue171019I}
\Box^{(q)}_b=\hat\Box^{(q)}_\beta+P,
\end{equation}
where $P: \Omega^{0,q}(X)\To\Omega^{0,q}(X)$ is a classical pseudodifferential operator of order one. From \eqref{e-gue171019I} and \eqref{e-gue171016j}, for $u\in\Omega^{0,q}_m(X)$, we have
\begin{equation}\label{e-gue171019II}
\begin{split}
\norm{(\frac{1}{m}\Box^{(q)}_b)u}_X&\leq \norm{(\frac{1}{m}\hat\Box^{(q)}_\beta)u}_X+\frac{1}{m}\norm{Pu}_X\\
&\leq C_1\Bigr( \norm{(\frac{1}{m}\hat\Box^{(q)}_\beta)u}_X+\frac{1}{m}\norm{u}_{1,X}\Bigr)\\
&\leq C_2\Bigr(\norm{(\frac{1}{m}\hat\Box^{(q)}_\beta)u}_X+\frac{1}{m}\norm{Tu}_{X}+\frac{1}{m}\norm{u}_X\Bigr)\\
&\leq C_3\Bigr(\norm{(\frac{1}{m}\hat\Box^{(q)}_\beta)u}_X+\norm{u}_{X}\Bigr),
\end{split}
\end{equation}
where $C_1>0$, $C_2>0$, $C_3>0$ are constants independent of $m$ and $u$. From \eqref{e-gue171019II}, we see that \eqref{e-gue171019} holds for $k=1$. We assume that \eqref{e-gue171019} holds for $k=k_0\in\mathbb N$. We are going to prove that \eqref{e-gue171019} holds for $k=k_0+1$. Fix $u\in\Omega^{0,q}_m(X)$. By the induction assumption, we have
\begin{equation}\label{e-gue171019III}
\norm{(\frac{1}{m}\Box^{(q)}_b)^{k_0+1}u}_X\leq C\Bigr(\sum^{k_0}_{\ell=0}\norm{(\frac{1}{m}\hat\Box^{(q)}_\beta)^\ell(\frac{1}{m}\Box^{(q)}_bu)}_X\Bigr),
\end{equation}
where $C>0$ is a constant independent of $m$ and $u$. Fix $\ell\in\set{0,1,\ldots,k_0}$. From \eqref{e-gue171016} and \eqref{e-gue171019I}, we have
\begin{equation}\label{e-gue171019a}
\begin{split}
(\frac{1}{m}\hat\Box^{(q)}_\beta)^\ell(\frac{1}{m}\Box^{(q)}_bu)&=\frac{1}{m^{\ell+1}}(\hat\Box^{(q)}_\beta)^\ell(\hat\Box^{(q)}_\beta+P)u\\
&=\frac{1}{m^{\ell+1}}(\hat\Box^{(q)}_\beta)^{\ell+1}u+\frac{1}{m^{\ell+1}}\sum^{\ell+1}_{j=1}T_j\circ(\hat\Box^{(q)}_\beta)^{\ell+1-j}u,
\end{split}
\end{equation}
where $T_j: \Omega^{0,q}(X)\To\Omega^{0,q}(X)$ is a classical pseudodifferential operator of order $j$, $j=1,\ldots,\ell+1$. From \eqref{e-gue171016j}, for every $j=1,\ldots,\ell+1$, we have
\begin{equation}\label{e-gue171019aI}
\begin{split}
&\norm{T_j\circ(\hat\Box^{(q)}_\beta)^{\ell+1-j}u}_X\\
&\leq C_1\norm{(\hat\Box^{(q)}_\beta)^{\ell+1-j}u}_{j,X}\\
&\leq C_2\Bigr(\norm{(\hat\Box^{(q)}_\beta)^{\ell+2-j}u}_{j-1,X}+m\norm{(\hat\Box^{(q)}_\beta)^{\ell+1-j}u}_{j-1,X}\Bigr)\\
&\leq C_3\Bigr(\norm{(\hat\Box^{(q)}_\beta)^{\ell+3-j}u}_{j-2,X}+m\norm{(\hat\Box^{(q)}_\beta)^{\ell+2-j}u}_{j-2,X}+m^2\norm{(\hat\Box^{(q)}_\beta)^{\ell+1-j}u}_{j-2,X}\Bigr)\\
&\leq C_4\Bigr(\sum^{\ell+1}_{k=j}m^k\norm{(\hat\Box^{(q)}_\beta)^{\ell+1-k}u}_{X}\Bigr),
\end{split}
\end{equation}
where $C_1>0$, $C_2>0$, $C_3>0$, $C_4>0$ are constants independent of $m$ and $u$.
From \eqref{e-gue171019a} and \eqref{e-gue171019aI}, we get \eqref{e-gue171019} for $k=k_0+1$. The theorem follows.
\end{proof}

We need

\begin{lemma}\label{t-gue171019I}
Fix $m\in\mathbb Z$. For every $k\in\mathbb N$, there is a constant $\hat C_k>0$ independent of $m$ such that
\begin{equation}\label{e-gue171019y}
\norm{(\frac{1}{m}\Box^{(q)}_b)^ku}^2_X\leq\hat C_k\Bigr(\sum^{2k-1}_{\ell=0}\norm{(\frac{1}{m}\hat\Box^{(q)}_\beta)^\ell u}_X\Bigr)
\Bigr(\norm{(\frac{1}{m}\hat\Box^{(q)}_\beta)u}_X+\frac{1}{\sqrt{m}}\norm{u}_X\Bigr),\ \ \forall u\in\Omega^{0,q}_m(X).
\end{equation}
\end{lemma}

\begin{proof}
We first claim that there is a constant $C>0$ independent of $m$ such that
\begin{equation}\label{e-gue171019pmy}
\norm{(\frac{1}{m}\Box^{(q)}_b)u}_X\leq C\Bigr(\norm{(\frac{1}{m}\hat\Box^{(q)}_\beta)u}_X+\frac{1}{\sqrt{m}}\norm{u}_X\Bigr),\ \ \forall u\in\Omega^{0,q}_m(X).
\end{equation}
From \eqref{e-gue171014aI}, we have $\ddbar_b=\hat\ddbar_\beta+r^{(q)}_1:\Omega^{0,q}(X)\To\Omega^{0,q+1}(X)$, $\ddbar^\star_b=\hat\ddbar^\star_\beta+r^{(q+1)}_2: \Omega^{0,q+1}(X)\To\Omega^{0,q}(X)$, where $r^{(q)}_1$ and $r^{(q+1)}_2$ are psudodifferential operators of order $0$. We have
\begin{equation}\label{e-gue171014pmyI}
\begin{split}
\Box^{(q)}_b&=\ddbar_b\,\ddbar^\star_b+\ddbar^\star_b\,\ddbar_b=(\hat\ddbar_\beta+r^{(q-1)}_1)(\hat\ddbar^\star_\beta+r^{(q)}_2)+(\hat\ddbar^\star_\beta+r^{(q+1)}_2)(\hat\ddbar_\beta+r^{(q)}_1)\\
&=\hat\Box^{(q)}_\beta+r^{(q-1)}_1\circ\hat\ddbar^\star_\beta+\hat\ddbar_\beta\circ r^{(q)}_2+
r^{(q+1)}_2\circ\hat\ddbar_\beta+\hat\ddbar^\star_\beta\circ r^{(q)}_1+r^{(q+1)}_2\circ r^{(q)}_1+r^{(q-1)}_1\circ r^{(q)}_2.
\end{split}
\end{equation}
Let $u\in\Omega^{0,q}_m(X)$. It is not difficult to check that
\begin{equation}\label{e-gue171018z}
\begin{split}
&\norm{(r^{(q-1)}_1\circ\hat\ddbar^\star_\beta)u}^2_X+\norm{(r^{(q+1)}_2\circ\hat\ddbar_\beta)u}^2_X+\norm{(r^{(q+1)}_2\circ r^{(q)}_1+r^{(q-1)}_1\circ r^{(q)}_2)u}^2_X\\
&\leq C\Bigr(\norm{\hat\ddbar^\star_\beta u}^2_X+\norm{\hat\ddbar_\beta u}^2_X+\norm{u^2}_X\Bigr)\\
&\leq C_1\Bigr(\norm{\hat\Box^{(q)}_\beta u}^2_X+\norm{u}^2_X\Bigr),
\end{split}
\end{equation}
where $C>0$, $C_1>0$ are constants independent of $m$ and $u$. We have
\begin{equation}\label{e-gue171018zI}
\begin{split}
&\norm{(\hat\ddbar_\beta\circ r^{(q)}_2)u}^2_X\leq \norm{(\hat\ddbar_\beta\circ r^{(q)}_2)u}^2_X+\norm{(\hat\ddbar^\star_\beta\circ r^{(q)}_2)u}^2_X\\
&\leq (\,(\hat\Box^{(q-1)}_\beta\circ r^{(q)}_2)u\,|\,r^{(q)}_2u\,)_X.\end{split}
\end{equation}
Now,
\begin{equation}\label{e-gue171018zII}
\begin{split}
&(\,(\hat\Box^{(q-1)}_\beta\circ r^{(q)}_2)u\,|\,r^{(q)}_2u\,)_X\\
&=(\,(r^{(q)}_2\circ\hat\Box^{(q)}_\beta)u\,|\,r^{(q)}_2u\,)_X+((\hat\Box^{(q-1)}_\beta\circ r^{(q)}_2-r^{(q)}_2\circ\hat\Box^{(q)}_\beta)u\,|\,r^{(q)}_2u\,)_X.
\end{split}
\end{equation}
Since the principal symbol of $\hat\Box^{(q)}_\beta$ is scalar, we deduce that the operator
\[\hat\Box^{(q-1)}_\beta\circ r^{(q)}_2-r^{(q)}_2\circ\hat\Box^{(q)}_\beta:\Omega^{0,q}(X)\To\Omega^{0,q-1}(X)\]
is a pseudodifferential operator of order $1$. From this observation, \eqref{e-gue171018zI}, \eqref{e-gue171018zII} and \eqref{e-gue171016j}, we have
\begin{equation}\label{e-gue171018zIII}
\begin{split}
\norm{(\hat\ddbar_\beta\circ r^{(q)}_2)u}^2_X&\leq C\Bigr(\norm{\hat\Box^{(q)}_\beta u}^2_X+\norm{u}^2_X+\norm{u}_{1,X}\norm{u}_X\Bigr)\\
&\leq C_1\Bigr(\norm{\hat\Box^{(q)}_\beta u}^2_X+\norm{u}^2_X+\norm{Tu}_{X}\norm{u}_X\Bigr)\\
&\leq C_2\Bigr(\norm{\hat\Box^{(q)}_\beta u}^2_X+\norm{u}^2_X+m\norm{u}^2_X\Bigr),
\end{split}
\end{equation}
where $C>0$, $C_1>0$ and $C_2>0$ are constants independent of $m$ and $u$. We can repeat the procedure \eqref{e-gue171018zIII} with minor change and get
\begin{equation}\label{e-gue171018za}
\norm{(\hat\ddbar^\star_\beta\circ r^{(q)}_1)u}^2_X\leq C_0\Bigr(\norm{\hat\Box^{(q)}_\beta u}^2_X+\norm{u}^2_X+m\norm{u}^2_X\Bigr),
\end{equation}
where $C_0>0$ is a constant independent of $m$ and $u$. From \eqref{e-gue171014pmyI}, \eqref{e-gue171018z}, \eqref{e-gue171018zIII} and \eqref{e-gue171018za},
we get \eqref{e-gue171019pmy}.

From \eqref{e-gue171019}, \eqref{e-gue171019pmy}, for $u\in\Omega^{0,q}_m(X)$, we have
\[\begin{split}
\norm{(\frac{1}{m}\Box^{(q)}_b)^ku}^2_X&=(\,(\frac{1}{m}\Box^{(q)}_b)^ku\,|\,(\frac{1}{m}\Box^{(q)}_b)^ku\,)_X\\
&=(\,(\frac{1}{m}\Box^{(q)}_b)^{2k-1}u\,|\,(\frac{1}{m}\Box^{(q)}_b)u\,)_X\\
&\leq\norm{(\frac{1}{m}\Box^{(q)}_b)^{2k-1}u}_X\norm{(\frac{1}{m}\Box^{(q)}_b)u}\\
&\leq\hat C_k\Bigr(\sum^{2k-1}_{\ell=0}\norm{(\frac{1}{m}\hat\Box^{(q)}_\beta)^\ell)u}_X\Bigr)
\Bigr(\norm{(\frac{1}{m}\hat\Box^{(q)}_\beta)u}_X+\frac{1}{\sqrt{m}}\norm{u}_X\Bigr),
\end{split}\]
where $\hat C_k>0$ is a constant independent of $m$ and $u$. We get \eqref{e-gue171019y}.
\end{proof}

We can repeat the proof of Lemma~\ref{t-gue171019I} and deduce

\begin{lemma}\label{t-gue171019II}
Fix $m\in\mathbb Z$. For every $k\in\mathbb N$, there is a constant $\Td C_k>0$ independent of $m$ such that
\begin{equation}\label{e-gue171019yI}
\norm{(\frac{1}{m}\hat\Box^{(q)}_\beta)^ku}^2_X\leq\Td C_k\Bigr(\sum^{2k-1}_{\ell=0}\norm{(\frac{1}{m}\Box^{(q)}_b)^\ell u}_X\Bigr)
\Bigr(\norm{(\frac{1}{m}\Box^{(q)}_b)u}_X+\frac{1}{\sqrt{m}}\norm{u}_X\Bigr),\ \ \forall u\in\Omega^{0,q}_m(X).
\end{equation}
\end{lemma}

\section{Index theorem} \label{s-gue171023}
In this section, we are devoted to proving Theorem ~\ref{t-gue171002}. It should be pointed out that in this section, we do not need any of the results stated in Section~\ref{s-gue200131yyd}.
Fix $m\in\mathbb Z$.
Let $\hat\ddbar_\beta:\Omega^{0,q}_m(X)\rightarrow\Omega^{0,q+1}_m(X)$ be the operator given by \eqref{e-gue171014}. Since $\hat\ddbar_\beta^2=0$,
we have $\hat\ddbar_\beta$-complex:
\begin{equation}\label{complexbm}
\cdot\cdot\cdot\rightarrow\Omega^{0,q-1}_{m}(X)\xrightarrow{\hat\ddbar_\beta}\Omega^{0,q}_{m}(X)
\xrightarrow{\hat\ddbar_\beta}\Omega^{0,q+1}_{m}(X)\rightarrow\cdot\cdot\cdot
\end{equation}
The $m$-th Fourier component of $\hat\ddbar_\beta$-cohomology group is given by
\begin{equation}\label{e-gue171023y}
\hat H^{q}_{\beta,m}(X):=\frac{{\rm Ker\,}\hat\ddbar_\beta: \Omega_m^{0, q}(X)\rightarrow\Omega_m^{0, q+1}(X)}{{\rm Im\,}\hat\ddbar_\beta: \Omega_m^{0, q-1}(X)\rightarrow\Omega_m^{0, q}(X)}.
\end{equation}
We can repeat the proof of Theorem 3.7 in~\cite{CHT} and deduce that
\begin{equation}\label{e-gue171023qp}
\hat H^{q}_{\beta,m}(X)\cong{\rm Ker\,}\hat\Box^{(q)}_{\beta,m},\ \ \forall q=0,1,\ldots,n-1,\ \ \forall m\in\mathbb Z.
\end{equation}

Let $\ddbar_b:\Omega^{0,q}_m(X)\rightarrow\Omega^{0,q+1}_m(X)$ be the tangential Cauchy-Riemann operator. We have $\ddbar_b$-complex:
\begin{equation}\label{complexbmb}
\cdot\cdot\cdot\rightarrow\Omega^{0,q-1}_{m}(X)\xrightarrow{\ddbar_b}\Omega^{0,q}_{m}(X)
\xrightarrow{\ddbar_b}\Omega^{0,q+1}_{m}(X)\rightarrow\cdot\cdot\cdot
\end{equation}
The $m$-th Fourier component of Kohn-Rossi cohomology group is given by
\begin{equation}\label{e-gue171023ya}
H^{q}_{b,m}(X):=\frac{{\rm Ker\,}\ddbar_b: \Omega_m^{0, q}(X)\rightarrow\Omega_m^{0, q+1}(X)}{{\rm Im\,}\ddbar_b: \Omega_m^{0, q-1}(X)\rightarrow\Omega_m^{0, q}(X)}.
\end{equation}
From \eqref{e-gue171014aII}, we know that $\hat\ddbar_{\beta}=\ddbar_{b}$+lower order terms. Since the index is homotopy invariant (see Theorem 4.7 in~\cite{CHT}), we deduce that
\begin{equation}\label{e-gue171023yaI}
\sum^{n-1}_{j=0}(-1)^j{\rm dim\,}H^{j}_{b,m}(X)=\sum^{n-1}_{j=0}(-1)^j{\rm dim\,}\hat H^{j}_{\beta,m}(X),\ \ \forall m\in\mathbb Z.
\end{equation}
From \eqref{e-gue171023yaI}, \eqref{e-gue171023qp} and Theorem~\ref{t-gue171014pm}, we conclude that for $m\gg1$, we have
\begin{equation}\label{e-gue171023yaII}
\sum^{n-1}_{j=0}(-1)^j{\rm dim\,}H^{j}_{m}(\ol M)=\sum^{n-1}_{j=0}(-1)^j{\rm dim\,}H^{j}_{b,m}(X).
\end{equation}
The formula for $\sum^{n-1}_{j=0}(-1)^j{\rm dim\,}H^{j}_{b,m}(X)$ was obtained by~\cite{CHT}. What follows is just a repetition of the results of ~\cite{CHT}.


\begin{definition}\label{d-gue50508d}
Let $D\subset X$ be an open set and $u\in C^\infty(D)$.
We say that $u$ is {\it rigid} if $Tu=0$, $u$ is Cauchy-Riemann (CR for short)
if $\overline\partial_bu=0$ and $u$ is a rigid CR function if $\overline\partial_bu=0$ and $Tu=0$.
\end{definition}

\begin{definition}
\label{d-gue150508dI} Let $F$ be a complex vector bundle of rank $r$ over $X$. We say
that $F$ is {\it rigid} (resp. CR) if $X$ can be covered by open subsets $U_j$ with
trivializing frames $\{f^1_j,f^2_j,\dots,f^r_j\}$ such that the corresponding transition functions
are rigid (resp. CR) (in the sense of the
preceding definition).
\end{definition}

\begin{definition}\label{d-gue180810}
Let $F$ be a complex vector bundle of rank $r$ over $X$. Fix open cover $(U_j)^N_{j=1}$ of $X$ and a family $\set{f^1_j,f^2_j,\dots,f^r_j}^N_{j=1}$ of trivializing frames $\set{f^1_j,f^2_j,\dots,f^r_j}$
on each $U_j$ such that the entries of the transition matrices between different frames $\set{f^1_j,f^2_j,\dots,f^r_j}$ are rigid (CR).
For any frame $\set{\hat f^1,\hat f^2,\dots,\hat f^r}$ of $F$ define on an open set $D$ of $X$, we say that $\set{\hat f_1,\hat f_2,\dots,\hat f_r}$ is rigid (CR) with respect to $\set{f^1_j,f^2_j,\dots,f^r_j}^N_{j=1}$ if all the entires of the transition matrices between $\set{\hat f^1,\hat f^2,\dots,\hat f^r}$ and $\set{f^1_j,f^2_j,\dots,f^r_j}$ are annihilated by $T$, for every $j=1,\ldots,N$, with $D\bigcap U_j\neq\emptyset$.
\end{definition}

Let $F$ be a rigid (CR) vector bundle over $X$. In this work, we fix open cover $(U_j)^N_{j=1}$ of $X$ and  a family $\set{f^1_j,f^2_j,\dots,f^r_j}^N_{j=1}$ of trivializing frames $\set{f^1_j,f^2_j,\dots,f^r_j}$ on each $U_j$ such that the entries of the transition matrices between different frames $\set{f^1_j,f^2_j,\dots,f^r_j}$ are rigid (CR). We say that a frame $f$ is rigid (CR) if $f$ is rigid (CR) with respect to $\set{f^1_j,f^2_j,\dots,f^r_j}^N_{j=1}$.
We can define the operator $T$ on $\Omega^{0,q}(X,F)$.
Indeed, every $u\in\Omega^{0,q}(X,F)$ can be written on $U_j$ as
$u=\sum u_\ell\otimes f^\ell_j$ and we set $Tu=\sum Tu_\ell\otimes f^\ell_j$.
Then $Tu$ is well defined as element of $\Omega^{0,q}(X,F)$,
since the entries of the transition matrices between different frames $\set{f^1_j,f^2_j,\dots,f^r_j}$
are annihilated by $T$.

Let $X$ be a compact CR manifold  with a locally free transversal CR $S^1$ action.
In this work, we say that a trivializing frame $f$ of $T^{1,0}X$ is rigid if $f$ is rigid with respect to BRT frames (see Lemma~\ref{t-gue150514} for the meaning of BRT frames).  By using duality, we can also define BRT frames for the bundle $\oplus^{2n-1}_{r=1}\Lambda^r(\Complex T^*X)$ and we say that a trivializing frame $f$ of $\oplus^{2n-1}_{r=1}\Lambda^r(\Complex T^*X)$ is rigid if $f$ is rigid with respect to BRT frames.

For $r=0,1,2,\ldots,2n-2$, put $\Omega^r_{0}(X)=\left\{u\in
\oplus_{p+q=r}\Omega^{p,q}(X);\, Tu=0\right\}$ and set $\Omega^%
\bullet_{0}(X)=\oplus_{r=0}^{2n-2}\Omega^r_{0}(X)$. Since $Td=dT$, we have $d$-complex:
\begin{equation*}
d:\cdots\rightarrow\Omega^{r-1}_{0}(X)\rightarrow\Omega^{r}_{0}(X)%
\rightarrow\Omega^{r+1}_{0}(X)\rightarrow\cdots
\end{equation*}
Define the $r$-th tangential de Rham cohomology group:
\begin{equation*}
\mathcal{H}%
^r_{b,0}(X):=\frac{\mathrm{Ker\,}d:\Omega^{r}_{0}(X)\rightarrow%
\Omega^{r+1}_{0}(X)}{\mathrm{Im\,}d:\Omega^{r-1}_{0}(X)\rightarrow%
\Omega^{r}_{0}(X)}.
\end{equation*}
Put $\mathcal{H}^\bullet_{b,0}(X)=\oplus_{r=0}^{2n-2}%
\mathcal{H}^r_{b,0}(X)$.

Let a complex vector bundle $F$ over $X$ of rank $r$ be {\it rigid} as in Definition
\ref{d-gue150508dI}. It was shown
in \cite[Theorem~2.12]{CHT} that there exists a connection $\nabla$ on $F$ such
that for any rigid local frame $f=(f_1,f_2,\ldots,f_r)$ of $F$ on an open
set $D\subset X$, the connection matrix $\theta(\nabla,f)=\left(\theta_{j,k}%
\right)^r_{j,k=1}$ satisfies
\begin{equation*}
\theta_{j,k}\in\Omega^1_{0}(D),
\end{equation*}
for $j,k=1,\ldots,r$. We call $\nabla$ as such a {\it rigid connection} on $F$. Let
\[\Theta(\nabla,F)\in C^\infty(X,\Lambda^2(\mathbb{C }T^*X)\otimes\mathrm{End\,%
}(F))\]
be the associated {\it tangential curvature}.

Let $h(z)=\sum^\infty_{j=0}a_jz^j$ be a
real power series on $z\in\mathbb{C}$. Set
\begin{equation*}
H(\Theta(\nabla,F))=\mathrm{Tr\,}\Bigr(h\bigr(\frac{i}{2\pi}\Theta(\nabla,F)%
\bigr)\Bigr).
\end{equation*}
It is clear that $H(\Theta(\nabla,F))\in\Omega^{*}_{0}(X)$ and is known that $H(\Theta(\nabla,F))\in\Omega^{*}_{0}(X)$ is a closed differential form and the tangential de Rham cohomology class
\begin{equation*}
[H(\Theta(\nabla,F))]\in\mathcal{H}^\bullet_{b,0}(X)
\end{equation*}
does not depend on the choice of rigid connections $\nabla$, cf. \cite[Theorem 2.6, Theorem 2.7]{CHT}.
For  $h(z)=e^{z}$ put
\begin{equation}  \label{e-gue160607}
\mathrm{ch_b\,}(\nabla,F):=H(\Theta(\nabla,F))\in\Omega^\bullet_{0}(X),
\end{equation}
and for
$h(z)=\log (\frac{z}{1-e^{-z}})$ set
\begin{equation}  \label{e-gue150607I}
\mathrm{Td_b\,}(\nabla,F):=e^{H(\Theta(\nabla,F))}\in\Omega^\bullet_{0}(X).
\end{equation}

We can now introduce tangential Todd
class and tangential Chern character.

\begin{definition}\label{d-gue150516}
The {\it tangential Chern character} of $F$ is given by
\begin{equation*}
\mathrm{ch_b\,}(F):=[\mathrm{ch_b\,}(\nabla,F)]\in\mathcal{H}%
^\bullet_{b,0}(X)
\end{equation*}
and the {\it tangential Todd class} of $F$ is given by
\begin{equation*}
\mathrm{Td_b\,}(F)=[\mathrm{Td_b\,}(\nabla,F)]\in\mathcal{H}%
^\bullet_{b,0}(X).
\end{equation*}
\end{definition}



Baouendi-Rothschild-Treves~\cite{BRT85} proved that $T^{1,0}X$ is a rigid
complex vector bundle over $X$. The tangential Todd class of
$T^{1,0}X$ is thus well defined.

In~\cite{CHT}, it was shown that for every $m\in\mathbb Z$,
\begin{equation}\label{e-gue171023qqb}
\sum_{j=0}^{n-1}(-1)^j{\rm dim\,}H^{q}_{b,m}(X)=\frac{1}{2\pi}\int_X \mathrm{Td_b\,}(T^{1,0}X) \wedge e^{-m\frac{d\omega_0}{2\pi}} \wedge(-\omega_0).
\end{equation}
From \eqref{e-gue171023qqb} and \eqref{e-gue171023yaII} we have the following Index theorem for $\overline\partial_m$.
\begin{theorem}\label{t-gue1710021}
	With the notations and assumptions above, there is a $m_0>0$ such that for every $m\in\mathbb Z$ with $m\geq m_0$, we have
	\begin{equation}\label{e-gue1710021}
	\sum_{j=0}^{n-1}(-1)^j{\rm dim\,}H^{j}_{m}(\ol M)=\frac{1}{2\pi}\int_X \mathrm{Td_b\,}(T^{1,0}X) \wedge e^{-m\frac{d\omega_0}{2\pi}} \wedge(-\omega_0),
	\end{equation}
	where $\mathrm{Td_b\,}(T^{1,0}X)$ denotes the \emph{tangential Todd class} of $T^{1,0}X$ (see Definition~\ref{d-gue150516}).
\end{theorem}

\section{The scaling technique}\label{s-gue171019}

The main goal of the rest of this paper is to prove Theorem~\ref{t-gue170930}, that is, to get Morse inequalities for $H^q_m(\ol M)$ when $m\To+\infty$. 
 In view of Theorem~\ref{t-gue171014pm}, we need to estimate ${\rm dim\,}{\rm Ker\,}\hat\Box^{(q)}_{\beta,m}$ when $m\To+\infty$. Let $u\in{\rm Ker\,}\hat\Box^{(q)}_{\beta,m}$, $\norm{u}_X=1$. From Lemma~\ref{t-gue171019I}, we see that $\norm{(\frac{1}{m}\Box^{(q)}_b)^ku}^2_X\leq C_k$, for every $k\in\mathbb N$, where $C_k>0$ is a constant independent of $m$ and $u$.  From this observation, we see that we can apply  the scaling technique for $\Box^{(q)}_b$ used in 
 ~\cite [Section 1.4]{HL15} to estimate pointwise norm of $u\in {\rm Ker\,}\hat\Box^{(q)}_{\beta,m}$ and study  ${\rm dim\,}{\rm Ker\,}\hat\Box^{(q)}_{\beta,m}$ when $m\To+\infty$ (see Section~\ref{s-gue171021}). For the convenience of the reader, 
 in this section, we will recall the scaling technique used in~\cite [Section 1.4]{HL15}.
 
We need the following result due to Baouendi-Rothschild-Treves~\cite{BRT85}.

\begin{lemma}\label{t-gue150514}
For every point $x_0\in X$, we can find local coordinates $x=(x_1,\cdots,x_{2n-1})=(z,\theta)=(z_1,\cdots,z_{n-1},\theta), z_j=x_{2j-1}+ix_{2j},j=1,\cdots,n-1, x_{2n-1}=\theta$, defined in some small neighborhood $D=\{(z, \theta)\in\Complex^{n-1}\times\Real;\, \abs{z}<\delta, -\varepsilon_0<\theta<\varepsilon_0\}$ of $x_0$, $\delta>0$, $0<\varepsilon_0<\pi$, such that $(z(x_0),\theta(x_0))=(0,0)$ and
\begin{equation}\label{e-can}
\begin{split}
&T=\frac{\partial}{\partial\theta},\\
&Z_j=\frac{\partial}{\partial z_j}+i\frac{\partial\varphi}{\partial z_j}(z)\frac{\partial}{\partial\theta},\ \ j=1,\cdots,n-1,
\end{split}
\end{equation}
where $Z_j(x), j=1,\cdots, n-1$, form a basis of $T_x^{1,0}X$, for each $x\in D$ and $\varphi(z)\in C^\infty(D,\mathbb R)$ independent of $\theta$.  We call $(D,x=(z,\theta),\varphi)$ BRT trivialization, $\set{Z_j}^{n-1}_{j=1}$ BRT frame and we call $x=(z,\theta)$ canonical coordinates.
\end{lemma}

\begin{remark}\label{r-gue171021}
It is well-known that (see Lemma 1.17 in~\cite{HL15}) if $x_0\in X_{{\rm reg\,}}$, the canonical coordinates $(z,\theta)$ introduced in Lemma~\ref{t-gue150514} can be defined on  $D=\{(z, \theta)\in\Complex^{n-1}\times\Real;\, \abs{z}<\delta, -\pi<\theta<\pi\}$, for some $\delta>0$. Recall that $X_{{\rm reg\,}}$ is given by \eqref{e-gue171021y}.
\end{remark}

Now, we fix $x_0\in X$. Let $(D,x=(z,\theta),\varphi)$ be a BRT trivialization such that $x(x_0)=0$, where $D=\{(z, \theta)\in\Complex^{n-1}\times\Real;\,
\abs{z}<\delta, -\varepsilon_0<\theta<\varepsilon_0\}$, $\delta>0$, $0<\varepsilon_0<\pi$. It is easy to see that we can take $\varphi$ and $(z, \theta)$ so that
\begin{equation}\label{e-gue171020}
\begin{split}
&\varphi(z)=\sum\limits_{j=1}^{n-1}\lambda_j|z_j|^2+O(|z|^3), \ \ \forall (z, \theta)\in D,\\
&\langle\,\frac{\pr}{\pr z_j}\,|\,\frac{\pr}{\pr z_k}\,\rangle=\delta_{j,k}+O(\abs{z}),\ \ j, k=1,\ldots,n-1,\ \ \forall (z,\theta)\in D,
\end{split}
\end{equation}
where $\{\lambda_j\}_{j=1}^{n-1}$ are the eigenvalues of $\mathcal{L}_{x_0}$ with respect to the given $S^1$-invariant Hermitian metric $\langle\,\cdot\,|\,\cdot\,\rangle$ on $X$. Then the volume form with respect to the $S^1$-invariant Hermitian metric $\langle\,\cdot\,|\,\cdot\,\rangle$ on $X$ is $dv_X=\lambda(z)dv(z)d\theta$ where $\lambda(z)$ is a positive smooth function on $D$ and $dv(z)=2^{n-1}dx_1\wedge\cdots\wedge dx_{2n-2}$. Let $\{e^j(z)\}_{j=1}^{n-1}$ be an orthonormal frame of $T^{*0,1}X$ over $D$ such that
\begin{equation}\label{e-gue171021a}
e^j(z)=d\ol z_j+O(|z|),\ \ j=1,\ldots,n-1,
\end{equation}
and let $\{\ol L_j\}_{j=1}^{n-1}\subset T^{0,1}X$ be the dual frame of $\{e^j(z)\}_{j=1}^{n-1}$. Thus
\[L_j(z)=\frac{\partial}{\partial z_j}+O(|z|),\ \ j=1,\ldots,n-1.\]
We will always identify $D$ with an open subset of $\mathbb C^{n-1}\times\mathbb R$. Put $\tilde D=\{z\in\mathbb C^{n-1};\, |z|<\delta\}.$ Then $\varphi(z)$ can be treated as a real smooth function on $\tilde D$. Let $\Omega^{0, q}(\tilde D)$ be the space of smooth $(0,q)$ forms on $\tilde D$ and let $\Omega_0^{0, q}(\tilde D)$ be the subspace of $\Omega^{0, q}(\tilde D)$ whose elements have compact support in $\tilde D$.

For $r>0$, let
$\tilde D_r=\{z\in\mathbb C^{n-1};\, |z|<r\}$. Here $|z|<r$ means that $|z_j|<r, \forall j=1, \cdots, n-1$. For $m\in\mathbb N$, let $F_m$ be the scaling map: $F_m(z)=(\frac{z_1}{\sqrt m},\ldots, \frac{z_{n-1}}{\sqrt m}), z\in \tilde D_{\log m}$. From now on, we
assume $m$ is sufficiently large such that $F_m(\tilde D_{\log m})\Subset \tilde D$. We define the scaled bundle $F_m^*T^{*0, q} \tilde D$ on $\tilde D_{\log m}$ to be the bundle whose fiber at
$z\in \tilde D_{\log m}$ is
\begin{equation}\label{e-gue171020q}
F_m^*T^{*0,q}\tilde D|_{z}=\left\{\sum\nolimits_{|J|=q}^\prime a_Je^J\left(\frac{z}{\sqrt m}\right);\, a_J\in\mathbb C, |J|=q, J~\text{strictly increasing}\right\},
\end{equation}
where for multiindex $J=(j_1,\ldots,j_q)$, $e^J:=e^{j_1}\wedge e^{j_2}\wedge\cdots\wedge e^{j_q}$.
We take a Hermitian metric $\langle\,\cdot\,|\,\cdot\,\rangle_{F_m^*}$ on $F_m^*T^{*0,q}\tilde D$ so that at each point $z\in \tilde D_{\log m}$,
\begin{equation}\label{e-gue171020qI}
\left\{e^J\left(\frac{z}{\sqrt m}\right);\, |J|=q, J~\text{strictly increasing}\right\}
\end{equation}
is an orthonormal frame for $F_m^*T^{*0,q}\tilde D$ on $\tilde D_{\log m}$.
Let $F_m^*\Omega^{0,q}(\tilde D_r)$ denote the space of smooth sections of $F_m^*T^{*0,q} \tilde D$ over $\tilde D_r$ and let $F_m^*\Omega^{0,q}_0(\tilde D_r)$ be the subspace of $F_m^*\Omega^{0,q}(\tilde D_r)$ whose elements have compact support in $\tilde D_r$. Here $r<\log m$. Given $f\in\Omega^{0,q}(\tilde D)$.
We write $f=\sideset{}{'}\sum\nolimits_{|J|=q}f_Je^J$, where $\sideset{}{'}\sum$ means that
the summation is performed only over strictly increasing multiindices. We define the scaled form $F_m^*f\in F_m^*\Omega^{0,q}(\tilde D_{\log m})$ by
\begin{equation}\label{j8}
F_m^*f=\sideset{}{'}\sum\nolimits_{|J|=q}f_J\left(\frac{z}{\sqrt m}\right)e^J\left(\frac{z}{\sqrt m}\right),\ \  z\in\tilde D_{\log m}.
\end{equation}
For brevity, we denote $F_m^*f$ by $f(\frac{z}{\sqrt m})$.

Let $\overline\partial: \Omega^{0, q}(\tilde D)\rightarrow\Omega^{0, q+1}(\tilde D)$ be the Cauchy-Riemann operator. Then there exists a scaled differential operator $\overline\partial_{(m)}: F_m^*\Omega^{0,q}(\tilde D_{\log m})\rightarrow F_m^*\Omega^{0,q+1}(\tilde D_{\log m})$ such that
\begin{equation}\label{l1}
\overline\partial_{(m)}F_m^* f=\frac{1}{\sqrt{m}}F_m^*(\overline\partial f),\ \ \forall f\in\Omega^{0,q}(F_m(\tilde D_{\log m})).
\end{equation}
Let $(\,\cdot\,|\,\cdot\,)_{2mF_m^*\varphi}$ be the weighted inner product on the space $F_m^*\Omega^{0,q}_0(\tilde D_{\log m})$ defined as follows
\begin{equation}\label{e-gue171020k}
(\,f\,|\,g\,)_{2mF_m^*\varphi}=\int_{\tilde D_{\log m}}\langle\,f\,|\,g\,\rangle_{F_m^*}e^{-2mF_m^*\varphi}\lambda(\frac{z}{\sqrt m})dv(z).
\end{equation}
Let $\overline\partial^\star_{(m)}: F_m^*\Omega^{0,q+1}(\tilde D_{\log m})\rightarrow F_m^*\Omega^{0,q}(\tilde D_{\log m})$
be the formal adjoint of $\overline\partial_{(m)}$ with respect to $(\,\cdot\,|\,\cdot\,)_{2mF_m^*\varphi}$.
We now define the scaled complex Laplacian $\Box^{(q)}_{(m)}:F_m^*\Omega^{0,q}(\tilde D_{\log m})\rightarrow F_m^*\Omega^{0,q}(\tilde D_{\log m})$ which is given by
$
\Box^{(q)}_{(m)}=\overline\partial_{(m)}^\star\overline\partial_{(m)}
+\overline\partial_{(m)}\overline\partial_{(m)}^\star.
$
We have the G\r{a}rding's inequality as follows.

\begin{proposition}\label{k2}
For every $r>0$ with $\tilde D_{2r}\subset \tilde D_{\log m}$ and $s\in\mathbb N_0$, there is a constant $C_{r, s}>0$ independent of $m$ and the point $x_0$ such that
\begin{equation}\label{e-gue171020z}
\|u\|^2_{2mF_m^*\varphi, s+2, \tilde D_{r}}\leq C_{s, r}\left(\|u\|^2_{2mF_m^*\varphi, 0,\tilde D_{2r}}+\|\Box^{(q)}_{(m)}u\|^2_{2mF_m^*\varphi, s, \tilde D_{2r}}\right)
\end{equation}
for all $ u\in F_m^*\Omega^{0, q}(\tilde D_{\log m}),$ where $\|u\|_{2mF_m^*\varphi, s, \tilde D_{r}}$ is the weighted Sobolev norm of order $s$ with respect to the weight function $2mF_m^*\varphi$ which is given by
\begin{equation}\label{e-gue171020zI}
\|u\|^2_{2mF_m^*\varphi, s, \tilde D_r}=\sum^\prime\nolimits_{\alpha\in\mathbb N_0^{2n-2}, |\alpha|\leq s, |J|=q}\int_{\tilde D_r}|\partial^{\alpha}_xu_J|^2e^{-2mF_m^*\varphi}\lambda(\frac{z}{\sqrt{m}})dv(z),
\end{equation}
where $u=\sum_{|J|=q}^\prime u_Je^J(\frac{z}{\sqrt m})\in F_m^*\Omega^{0, q}(\tilde D_{\log m})$.
\end{proposition}

For $p\in X$, let $\det\mathcal{L}_p=\mu_1\cdots\mu_{n-1}$, where $\mu_j$, $j=1,\ldots,n-1$, are the eigenvalues of $\mathcal{L}_p$ with respect to $\langle\,\cdot\,|\,\cdot\,\rangle$.
We recall the following (see the proof of Theorem 2.1 in~\cite{HL15})

\begin{lemma}\label{t-gue171020w}
Let
\[h_m=\sideset{}{'}\sum_{\abs{J}=q}h_{m,J}(z)e^J(\frac{z}{\sqrt{m}})\in F^*_m\Omega^{0,q}(\tilde D_{\log m})\]
with $\norm{h_m}_{2mF_m^*\varphi, 0,\tilde D_{\log m}}\leq 1$, $m=1,2,3,\ldots$. Assume that
\[\lim_{m\To+\infty}\norm{\Box^{(q)}_{(m)}h_m}_{2mF_m^*\varphi, 0,\tilde D_{\log m}}=0\] and for every $k\in\mathbb N$, there is a constant $C_k>0$ independent of $m$ such that
\[\norm{(\Box^{(q)}_{(m)})^kh_m}_{2mF_m^*\varphi, 0,\tilde D_{\log m}}\leq C_k,\ \ \forall m=1,2,3,\ldots.\]
Then, for every strictly increasing multiindex $J$, $\abs{J}=q$, we have
\[\limsup_{m\To+\infty}\abs{h_{m,J}(0)}^2\leq\frac{1}{\pi^{n-1}}\mathrm 1_{X(q)}(x_0)|\det \mathcal{L}_{x_0}|\delta_{-}(J),\]
where $\delta_-(J)=1$ if $\lambda_j<0$, for every $j\in J$ and $\delta_-(J)=0$ otherwise. Here $\lambda_j$, $j=1,\ldots,n-1$, are the eigenvalues of  $\mathcal{L}_{x_0}$ with respect to
$\langle\,\cdot\,|\,\cdot\,\rangle$. Recall that $X(q)$ is given by \eqref{e-gue170930I}.
\end{lemma}


\section{Holomorphic Morse inequalities on complex manifolds with boundary}\label{s-gue171021}
In this section, we will prove Theorem~\ref{t-gue170930}.
\subsection{Proof of weak Morse inequalities}
In view of Theorem~\ref{t-gue171014pm}, we know that there is a $\hat m_0>0$ such that for all $m\geq\hat m_0$,
${\rm Ker\,}\hat\Box^{(q)}_{\beta,m}\cong H^q(\ol M)$. From now on, we assume that $m\geq\hat m_0$.

Let $\{f_{1},...,f_{d_{m}}\}$ be an orthonormal basis of ${\rm Ker\,}\hat\Box^{(q)}_{\beta,m}$.
Set $\hat \Pi_{\beta,m}^{(q)}(x)=\sum_{j=1}^{d_{m}}|f_{j}(x)|^{2}$.
Then
\begin{equation}\label{e-gue171022ka}
{\rm dim\,}H^q(\ol M)={\rm dim\,}{\rm Ker\,}\hat\Box^{(q)}_{\beta,m}=\int_{X}\hat \Pi_{\beta,m}^{(q)}(x)dv_{X}.
\end{equation}
Now, fix $x_0\in X$ and let $(D,x=(z,\theta),\varphi)$ be a BRT trivialization such that $x(x_0)=0$, where
\begin{equation}\label{e-gue171021t}
D=\{(z, \theta)\in\Complex^{n-1}\times\Real;\,
\abs{z}<\delta, -\varepsilon_0<\theta<\varepsilon_0\},\ \ \delta>0,\ \ 0<\varepsilon_0\leq\pi.
\end{equation}
We take $\varphi$ and $z$ so that \eqref{e-gue171020} hold.
 Let $\{e^j(z)\}_{j=1}^{n-1}$ be an orthonormal frame of $T^{*0,1}X$ over $D$ such that \eqref{e-gue171021a} hold. We will use the same notations as in Section~\ref{s-gue171019}.
For any $\alpha\in{\rm Ker\,}\hat\Box^{(q)}_{\beta,m}$, put $\alpha=\sideset{}{'}\sum_{\abs{J}=q}\alpha_Je^J$. For any strictly multiindex $J$, $\abs{J}=q$, put
\begin{equation}\label{e-gue171021b}
S_{m,J}^{(q)}(x)=\sup_{\alpha\in{\rm ker\,}\hat\Box^{(q)}_{\beta,m},\|\alpha\|_X=1}|\alpha_{J}|^{2}.
\end{equation}
The following is well-known (see Lemma 2.1 in~\cite{HM12} )

\begin{lemma}\label{l-gue171021}
We have
$\hat\Pi_{\beta,m}^{(q)}(x_0)=\sum_{|J|=q}^{'}S_{m,J}^{(q)}(x_0)$.
\end{lemma}

We can now prove

\begin{theorem}\label{Morsea}
There exists a constant $C>0$ independent of $x_0$ and $m$ such that
\begin{equation}\label{e-gue171021bI}
m^{-(n-1)}\hat\Pi_{\beta,m}^{(q)}(x_0)\leq C.
\end{equation}
If $x_0\in X_{{\rm reg\,}}$, we have
\begin{equation}\label{e-gue171021bII}
\limsup_{m\rightarrow\infty}m^{-(n-1)}\hat\Pi_{\beta,m}^{(q)}(x_0)\leq\frac{1}{2\pi^{n}}|\det\mathcal{L}_{x_0}|\cdotp
\mathrm{1}_{X(q)}(x_0),
\end{equation}
where $X_{{\rm reg\,}}$ is given by \eqref{e-gue171021y}.
\end{theorem}

\begin{proof}
Let $u\in{\rm Ker\,}\hat\Box^{(q)}_{\beta,m}$ with $\norm{u}_X=1$. On $D$,
write $u(z,\theta)=\hat u(z)e^{im\theta}$, $\hat u(z)\in\Omega^{0,q}(\tilde D)$, where $\tilde D=\set{z\in\Complex^{n-1};\, \abs{z}<\delta}$. Put
\begin{equation}\label{e-gue171021p}
\begin{split}
&v(z)=\hat u(z)e^{m\varphi},\\
&v^{(m)}=m^{-\frac{n-1}{2}}v(\frac{z}{\sqrt{m}}).
\end{split}
\end{equation}
It is easy to see that
\begin{equation}\label{e-gue171021pa}
\norm{v^{(m)}}^2_{2mF^*_m\varphi, 0,\tilde D_{\log m}}\leq\frac{1}{2\delta},
\end{equation}
where $\delta>0$ is as in \eqref{e-gue171021t}.
For every $k\in\mathbb N$, put
\begin{equation}\label{e-gue171021pI}
g_{k,m}(z):=e^{-im\theta}e^{m\varphi(z)}(\frac{1}{m}\Box_{b,m}^{(q)})^ku\in\Omega^{0,q}(\tilde D).
\end{equation}
From Lemma 2.11 in~\cite{HL15}, it is not difficult to see that for every $k\in\mathbb N$, we have
\begin{equation}\label{e-gue171021pII}
(\Box^{(q)}_{(m)})^kv^{(m)}=m^{-\frac{n-1}{2}}g_{k,m}(\frac{z}{\sqrt{m}}).
\end{equation}
From \eqref{e-gue171021pII}, it is straightforward to see that for every $k\in\mathbb N$, we have
\begin{equation}\label{e-gue171021pIII}
\norm{(\Box^{(q)}_{(m)})^kv^{(m)}}^2_{2mF^*_m\varphi, 0,\tilde D_{\log m}}\leq\frac{1}{2\delta}\norm{(\frac{1}{m}\Box_{b,m}^{(q)})^ku}^2_X.
\end{equation}
From \eqref{e-gue171021pIII}, \eqref{e-gue171019y} and note that $\hat\Box^{(q)}_\beta u=0$, we deduce that there is a constant $C_k>0$ independent of $m$, $u$ and the point $x_0$ such that
\begin{equation}\label{e-gue171021g}
\norm{(\Box^{(q)}_{(m)})^kv^{(m)}}_{2mF^*_m\varphi, 0,\tilde D_{\log m}}\leq C_k.
\end{equation}
Fix $r, r'<\log m$. Then by Proposition~\ref{k2}, \eqref{e-gue171021pa} and \eqref{e-gue171021g}, we have

\begin{equation}\label{e-gue171021gI}
\begin{split}
\|v^{(m)}\|_{2mF_{m}^{\ast}\varphi,s+2,D_{r}}&\leq C_{r,s}
\bigl(\|\Box_{(m)}^{(q)}v^{(m)}\|_{2mF_{m}^{\ast}\varphi,s,D_{r'}}+
\|v^{(m)}\|_{2mF_{m}^{\ast}\varphi,0,D_{r'}} \bigr)\\
&\leq C_{r,s}\sum_{j=0}^{s}\|(\Box_{(m)}^{(q)})^{j}v^{(m)}\|_{2mF_{m}^{\ast}\varphi,0,D_{r'}}\\
&\leq \tilde C_{r,s},
\end{split}
\end{equation}
where $C_{r,s}>0$ and $\tilde C_{r,s}>0$ are constants independent of $m$, $u$ and the point $x_0$.
From \eqref{e-gue171021gI} and by Sobolev embedding theorem, we have
\begin{equation}\label{e-gue171021gII}
m^{-(n-1)}|u(0)|^{2}=\abs{v^{(m)}(0)}^2\leq \|v^{(m)}\|_{2mF_{m}^{\ast}\varphi,n+2,D_{r}}\leq\hat C
\end{equation}
where $\hat C>0$ is a constant independent of $m$, $u$ and the point $x_0$. From \eqref{e-gue171021gII} and Lemma~\ref{l-gue171021},
we get the conclusion of the first part of the theorem.

Fix $J_0$ with $|J_0|=q$, $J_0=\{j_{1},...,j_{q}\}$, $j_{1}<\cdots<j_{q}$. By definition, there is a sequence
\[u_{m_\ell}=\sideset{}{'}\sum_{\abs{J}=q}u_{m_\ell,J}e^J\in{\rm Ker\,}\hat\Box^{(q)}_{\beta,m}\]
with $\norm{u_{m_\ell}}_X=1$, $\hat m_0\leq m_1<m_2<\cdots$, such that
\begin{equation}\label{e-gue171021s}
\limsup_{m\rightarrow\infty}m^{-(n-1)}S^{(q)}_{m,J_0}(x_{0})=
\lim_{\ell\rightarrow\infty}m_{\ell}^{-(n-1)}|u_{m_{\ell},J_0}(x_{0})|^{2}.
\end{equation}
On $D$,
write $u_{m_\ell}(z,\theta)=\hat u_{m_\ell}(z)e^{im_\ell\theta}=\sideset{}{'}\sum_{\abs{J}=q}e^{im_\ell\theta}\hat u_{m_\ell,J}e^J$, $\hat u_{m_\ell}(z)\in\Omega^{0,q}(\tilde D)$.
Put
\begin{equation}\label{e-gue171021ab}
\begin{split}
v_{m_\ell}(z)&=\hat u_{m_\ell}(z)e^{m_\ell\varphi}=\sideset{}{'}\sum_{\abs{J}=q}v_{m_\ell,J}e^J,\\
v^{(m_\ell)}&=m_\ell^{-\frac{n-1}{2}}v_{m_\ell}(\frac{z}{\sqrt{m_\ell}})=\sideset{}{'}\sum_{\abs{J}=q}m_\ell^{-\frac{n-1}{2}}v_{m_\ell,J}(\frac{z}{\sqrt{m_\ell}})e^J(\frac{z}{\sqrt{m_\ell}})\\
&=\sideset{}{'}\sum_{\abs{J}=q}v^{(m_\ell)}_J(z)e^J(\frac{z}{\sqrt{m_\ell}}).
\end{split}
\end{equation}
Assume that $x_0\in X_{{\rm reg\,}}$. In view of Remark~\ref{r-gue171021}, we can take $\delta>0$ in \eqref{e-gue171021t} to be $\pi$ and we have
\begin{equation}\label{e-gue171021abI}
\norm{v^{(m_\ell)}}^2_{2m_\ell F^*_{m_\ell}\varphi, 0,\tilde D_{\log m_\ell}}\leq\frac{1}{2\pi}. \
\end{equation}
As \eqref{e-gue171021pIII}, for every $k\in\mathbb N$, we have
\begin{equation}\label{e-gue171021abII}
\norm{(\Box^{(q)}_{(m_\ell)})^kv^{(m_\ell)}}_{2m_\ell F^*_{m_\ell}\varphi, 0,\tilde D_{\log m_\ell}}\leq\norm{(\frac{1}{m_\ell}\Box_{b,m_\ell}^{(q)})^ku_{m_\ell}}_X,\ \ \forall m_\ell.
\end{equation}
From \eqref{e-gue171021abII}, \eqref{e-gue171019y} and note that $\hat\Box^{(q)}_\beta u_{m_\ell}=0$, we have
\begin{equation}\label{e-gue171021d}
\lim_{\ell\To+\infty}\norm{\Box^{(q)}_{(m_\ell)}v^{(m_\ell)}}_{2m_\ell F_{m_\ell}^*\varphi, 0,\tilde D_{\log m_\ell}}=0,
\end{equation}
and for every $k\in\mathbb N$, there is a constant $C_k>0$ independent of $m_\ell$ such that
\begin{equation}\label{e-gue171021dI}
\norm{\bigl(\Box^{(q)}_{(m_\ell)}\bigr)^k v^{(m_\ell)}}_{2m_\ell F_{m_\ell}^*\varphi, 0,\tilde D_{\log m_\ell}}\leq C_k,\ \ \forall m_\ell.
\end{equation}
From \eqref{e-gue171021abI}, \eqref{e-gue171021d}, \eqref{e-gue171021dI} and Lemma~\ref{t-gue171020w}, we deduce that
\begin{equation}\label{e-gue171022}
\limsup_{m\rightarrow\infty}m^{-(n-1)}S^{(q)}_{m,J_0}(x_{0})=
\lim_{\ell\To+\infty}\abs{v^{(m_\ell)}_{J_0}(x_0)}^2\leq \frac{1}{2\pi^n}\mathrm 1_{X(q)}(x_0)|\det \mathcal{L}_{x_0}|\delta_{-}(J).
\end{equation}
From \eqref{e-gue171022} and Lemma~\ref{l-gue171021}, we
deduce that
\begin{equation}\label{e-gue171022k}
\begin{split}
&\limsup_{m\rightarrow\infty}m^{-(n-1)}\hat\Pi_{\beta,m}^{(q)}(x_0)\\
&\leq\sideset{}{'}\sum_{\abs{J}=q}\limsup_{m\rightarrow\infty}m^{-(n-1)}S^{(q)}_{m,J}(x_{0})\\
&\leq\sideset{}{'}\sum_{\abs{J}=q}\frac{1}{2\pi^n}\mathrm 1_{X(q)}(x_0)|\det \mathcal{L}_{x_0}|\delta_{-}(J)\\
&\leq \frac{1}{2\pi^n}\mathrm 1_{X(q)}(x_0)|\det \mathcal{L}_{x_0}|.
\end{split}
\end{equation}
From \eqref{e-gue171022k}, we get \eqref{e-gue171021bII}.
\end{proof}



From Theorem~\ref{Morsea}, \eqref{e-gue171022ka} and Fatou's Lemma, we get weak Morse inequalities \eqref{e-gue170930II}.

\subsection{Proof of strong Morse inequalities}\label{s-gue200123}
In the rest of this section, we will prove \eqref{e-gue170930III} and complete the proof of Theorem~\ref{t-gue170930}. We first introduce some notations.
For every $m\in\mathbb Z$, we extend $\hat\Box^{(q)}_{\beta, m}$ to $L^2_{(0,q),m}(X)$ by
\begin{equation}\label{j2}
\hat\Box^{(q)}_{\beta, m}:{\rm Dom}\hat\Box^{(q)}_{\beta,m}\subset L^2_{(0,q),m}(X)\rightarrow L^2_{(0,q),m}(X),
\end{equation}
where ${\rm Dom}\Box^{(q)}_{\beta,m}=\{u\in L^2_{(0,q),m}(X);\, \hat\Box^{(q)}_{\beta, m}u\in L^2_{(0,q),m}(X)~~\text{in the sense of distributions}\}$.
The following is well-known (see Section 3 in~\cite{CHT})

\begin{lemma}\label{gI}
Fix $m\in\mathbb Z$. Then, $\hat\Box^{(q)}_{\beta, m}: {\rm Dom\,}\hat\Box^{(q)}_{\beta, m}\subset L^2_{(0,q), m}(X)\rightarrow L^2_{(0,q), m}(X)$ is a self-adjoint operator, the spectrum of $\hat\Box^{(q)}_{\beta, m}$ denoted by
${\rm Spec\,}\hat\Box^{(q)}_{\beta, m}$ is a discrete subset of $[0,\infty)$. For every $\lambda\in{\rm Spec\,}\hat\Box^{(q)}_{\beta, m}$, $\lambda$ is an eigenvalue of $\hat\Box^{(q)}_{\beta,m}$ and  the eigenspace 
\begin{equation}\label{h}
\hat H^q_{\beta, m,\lambda}(X)=\set{u\in {\rm Dom\,}\hat\Box^{(q)}_{\beta, m};\, \hat\Box^{(q)}_{\beta, m}u=\lambda u}
\end{equation}
is finite dimensional with $\hat H^q_{\beta, m,\lambda}(X)\subset\Omega^{0,q}_m(X)$.
\end{lemma}

For every $\lambda>0$, put $\hat H^q_{\beta,m,\leq\lambda}(X):=\oplus_{\mu\in{\rm Spec\,}\hat\Box^{(q)}_{\beta,m}, 0\leq\mu\leq\lambda}\hat H^q_{\beta,m,\mu}(X)$. Let $\{g_{1},...,g_{a_{m}}\}$ be an orthonormal basis of $\hat H^q_{\beta,m,\leq\lambda}(X)$.
Set $\hat \Pi_{\beta,m,\leq\lambda}^{(q)}(x)=\sum_{j=1}^{a_{m}}|g_{j}(x)|^{2}$. We are going to get asymptotic leading term for $\hat \Pi_{\beta,m,\leq m\nu_m}^{(q)}(x)$, where $\nu_m$ is some sequence with $\lim_{m\To+\infty}\nu_m=0$. We need the following which is well-known (see the proof of Proposition 2.13 in \cite{HL15}).

\begin{proposition}\label{nn}
For any $x_0\in X(q)\cap X_{\rm reg}$, there exists a sequence $\alpha_m\in\Omega^{0, q}_m(X)$ such that

\begin{enumerate}
  \item $\lim\limits_{m\rightarrow\infty}m^{-(n-1)}|\alpha_m(x_0)|^2=\frac{1}{2\pi^n}|\det \mathcal L_{x_0}|.$
  \item $\lim_{m\rightarrow\infty}\|\alpha_m\|^2_X=1.$
  \item $\lim_{m\rightarrow\infty}\left\|\left(m^{-1}\Box^{(q)}_{b, m}\right)^k\alpha_m\right\|_X=0, \forall k\in\mathbb N.$
  \item $\text{There exists a sequence}~ \delta_m~\text{independent of}~x_0~\text{with}~\delta_m\rightarrow
0~\text{such that}\\
\norm{m^{-1}\Box^{(q)}_{b, m}\alpha_m}_X\leq\delta_m$.
\end{enumerate}
\end{proposition}

From Proposition~\ref{nn} and \eqref{e-gue171019yI}, we get

\begin{proposition}\label{nn1}
For any $x_0\in X(q)\cap X_{\rm reg}$, let $\alpha_m$ be given as in Proposition \ref{nn}. Then
\begin{enumerate}
  \item $\lim_{m\rightarrow\infty}\left\|\left(m^{-1}\hat\Box^{(q)}_{\beta, m}\right)^k\alpha_m\right\|_X=0, \forall k\in\mathbb N.$
  \item $\text{There exists a sequence}~ \tilde\delta_m~\text{independent of}~x_0~\text{with}~\tilde \delta_m\rightarrow
0~\text{such that}\\
\left(m^{-1}\hat\Box^{(q)}_{\beta, m}\alpha_m\big|\alpha_m\right)_X\leq\tilde\delta_m.$
\end{enumerate}
\end{proposition}

Now we are in a position to prove the following local strong Morse inequalities.

\begin{theorem}\label{o}
For any sequence $\nu_m>0$ with $\nu_m\rightarrow0$ as $m\rightarrow+\infty$, there exists a constant $C>0$ independent of $m$ and $x\in X$ such that
\begin{equation}\label{e-gue170919}
m^{-(n-1)}\hat\Pi^{q}_{\beta, m, \leq m\nu_m}(x)\leq C,\ \ \forall m\in\mathbb N,\ \ \forall x\in X.
\end{equation}

Moreover, there is a sequence $\tilde\delta_m>0$ with $\tilde\delta_m\rightarrow0$ as $m\rightarrow\infty$, such that for any sequence $\nu_m>0$ with $\lim_{m\To+\infty}\nu_m=0$ and $\lim\limits_{m\rightarrow+\infty}\frac{\tilde \delta_m}{\nu_m}=0$, we have
\begin{equation}\label{p}
\lim\limits_{m\rightarrow+\infty}m^{-(n-1)}\hat\Pi^{q}_{\beta, m, \leq m\nu_m}(x)=\frac{1}{2\pi^n}|\det\mathcal L_x|\cdot1_{X(q)}(x),\ \  \forall x\in X_{{\rm\, reg}}.
\end{equation}
\end{theorem}

\begin{proof}
The proof of \eqref{e-gue170919} is essentially the same as the proof of \eqref{e-gue171021bI}. We only need to prove \eqref{p}. Fix $x_0\in X_{{\rm reg\,}}$. We can repeat the proof of \eqref{e-gue171021bII} with minor change and get that for any sequence $\nu_m>0$ with $\nu_m\rightarrow0$ as $m\rightarrow+\infty$, we have
\begin{equation}\label{e-gue171021bIIq}
\limsup_{m\rightarrow\infty}m^{-(n-1)}\hat\Pi^{(q)}_{\beta,m,\leq m\nu_m}(x_0)\leq\frac{1}{2\pi^{n}}|\det\mathcal{L}_{x_0}|\cdotp
\mathrm{1}_{X(q)}(x_0).
\end{equation}
Hence, we only need to consider $x_0\in X(q)$.  Now, assume that $x_0\in X(q)$.
Let $\tilde\delta_m>0$ be a sequence as in Proposition~\ref{nn1} and let $\nu_m$ be any sequence with $\lim_{m+\infty}\nu_m=0$ and $\lim\limits_{m\rightarrow+\infty}\frac{\tilde \delta_m}{\nu_m}=0$.
Denote by $\hat H_{\beta,m,>m\nu_{m}}^{q}(X)$ the linear span of $\hat H_{\beta,m,\lambda}^{q}(X)$ with $\lambda>m\nu_{m}$. It is clear that
\[L_{(0,q),m}^{2}(X)=\hat H_{\beta,m,\leq m\nu_{m}}^{q}(X)\oplus
\ol{\hat H_{\beta,m,> m\nu_{m}}^{q}(X)},\]
where $\ol{\hat H_{\beta,m,> m\nu_{m}}^{q}(X)}$ denotes the $L^2$ closure of $\hat H_{\beta,m,> m\nu_{m}}^{q}(X)$. Let $\alpha_m\in\Omega^{0,q}_m(X)$ be as in Proposition~\ref{nn}. We have $\alpha_{m}=\alpha_{m, 1}+\alpha_{m, 2}$, where
$\alpha_{m, 1}\in\hat H_{\beta,m,\leq m\nu_{m}}^{q}(X)$, $\alpha_{m, 2}\in \ol{\hat H_{\beta, m ,> m\nu_{m}}^{q}(X)}$. By property (b) in Proposition~\ref{nn1}, we have
\begin{equation}\label{e-gue171022w}
\begin{split}
(\,\alpha_{m, 2}\,|\,\alpha_{m, 2}\,)_X&\leq \frac{1}{m\nu_{m}}(\,\hat\Box_{\beta,m}^{(q)}\alpha_{m, 2}\,|\,\alpha_{m, 2}\,)_X\\
&\leq\frac{1}{m\nu_{m}}(\,\hat\Box_{\beta,m}^{(q)}\alpha_{m}\,|\,\alpha_{m}\,)_X\leq \frac{\tilde \delta_{m}}{\nu_{m}}\rightarrow 0.
\end{split}
\end{equation}
Thus
\begin{equation}\label{e-gue171102cw}
\|\alpha_{m,2}\|_X\rightarrow 0
\end{equation}
and hence
\begin{equation}\label{e-gue171022v}
\|\alpha_{m,1}\|_X\rightarrow 1.
\end{equation}
Now, we claim that
\begin{equation}\label{e-gue171022vI}
m^{-(n-1)}|\alpha_{m, 2}(x_0)|^{2}\rightarrow 0.
\end{equation}
Let $(D,x=(z,\theta),\varphi)$ be a BRT trivialization such that $x(x_0)=0$ and \eqref{e-gue171020} hold, where
\[
D=\{(z, \theta)\in\Complex^{n-1}\times\Real;\,
\abs{z}<\delta, -\pi<\theta<\pi\},\ \ \delta>0.\]
Let $\alpha_{m, 2}=\hat\alpha_{m, 2}(z)e^{im\theta}$, $\beta_{m, 2}=\hat\alpha_{m, 2}(z)e^{m\varphi(z)}$.
Then
\begin{equation}\label{e-gue171022vII}
\lim_{m\rightarrow\infty}m^{-(n-1)}|\alpha_{m, 2}(0)|^{2}=
\lim_{m\rightarrow\infty}m^{-(n-1)}|\hat\alpha_{m, 2}(0)|^{2}=
\lim_{m\rightarrow\infty}m^{-(n-1)}|\beta_{m, 2}(0)|^{2}.
\end{equation}
We write
$\beta_{(m), 2}(z)=m^{-\frac{n-1}{2}}\beta_{m,2}(\frac{z}{\sqrt{m}})$.
By \eqref{e-gue171020z} and Sobolev embedding theorem, we have
\begin{equation}\label{e-gue171022vIII}
\begin{split}
&m^{-(n-1)}|\alpha_{m, 2}(0)|^{2}=|\beta_{(m), 2}(0)|^{2}\\
&\leq C\bigl(\|\beta_{(m),2}(z)\|_{2mF_{m}^{\ast}\varphi,0,D_{r}}^{2} +
\sum^{n+1}_{j=1}\|(\Box_{(m)}^{(q)})^j\beta_{(m), 2}(z)\|_{2mF_{m}^{\ast}\varphi,0,D_{r}}^{2}\bigr),
\end{split}
\end{equation}
where $r>0$ and $C>0$ are constants independent of $m$. Since $\|\alpha_{m,2}\|_X\rightarrow 0$, we have
\begin{equation}\label{e-gue171022vaI}
\|\beta_{(m),2}(z)\|_{2mF_{m}^{\ast}\varphi,D_{r}}^{2} \rightarrow 0.
\end{equation}
Moreover, as in the proof of Theorem~\ref{Morsea}, it is straightforward to check that for every $k\in\mathbb N$, we have
\begin{equation}\label{e-gue171022vaII}
\|(\Box_{(m)}^{(q)})^k\beta_{(m), 2}(z)\|_{2mF_{m}^{\ast}\varphi,0,D_{r}}^{2}\leq\frac{1}{2\pi}\norm{(\frac{1}{m}\Box^{(q)}_b)^k\alpha_{m,2}}_X.
\end{equation}
Fix $k\in\mathbb N$. From \eqref{e-gue171022vaII}, property (a) in Proposition~\ref{nn1}, \eqref{e-gue171019y} and \eqref{e-gue171102cw}, we have
\begin{equation}\label{e-gue171022vaIII}
\begin{split}
&\|(\Box_{(m)}^{(q)})^k\beta_{(m), 2}(z)\|_{2mF_{m}^{\ast}\varphi,0,D_{r}}^{2}\\
&\leq\frac{1}{2\pi}\norm{(\frac{1}{m}\Box^{(q)}_b)^k\alpha_{m,2}}_X\\
&\leq C_k\Bigr(\sum^{2k-1}_{\ell=0}\norm{(\frac{1}{m}\hat\Box^{(q)}_\beta)^\ell\alpha_{m,2}}_X\Bigr)
\Bigr(\norm{(\frac{1}{m}\hat\Box^{(q)}_\beta)\alpha_{m,2}}_X+\frac{1}{\sqrt{m}}\norm{\alpha_{m,2}}_X\Bigr)\\
&\leq C_k\Bigr(\sum^{2k-1}_{\ell=0}\norm{(\frac{1}{m}\hat\Box^{(q)}_\beta)^\ell\alpha_m}_X\Bigr)
\Bigr(\norm{(\frac{1}{m}\hat\Box^{(q)}_\beta)\alpha_m}_X+\frac{1}{\sqrt{m}}\norm{\alpha_{m,2}}_X\Bigr)\\
&\To0\ \ \mbox{as $m\To+\infty$},
\end{split}
\end{equation}
where $C_k>0$ is a constant independent of $m$.
From \eqref{e-gue171022vaIII}, \eqref{e-gue171022vaI} and \eqref{e-gue171022vIII}, the claim \eqref{e-gue171022vI} follows.

From \eqref{e-gue171022vI} and property (a) in Proposition~\ref{nn}, we deduce that
\begin{equation}\label{e-gue171022vw}
m^{-(n-1)}|\alpha_{m, 1}(x_0)|^{2}\rightarrow\frac{1}{2\pi^{n}}|\det \mathcal{L}_{x_0}|.
\end{equation}
Since
\begin{equation}\label{e-gue171022vwI}
m^{-(n-1)}\hat\Pi_{\beta,m,\leq m\nu_{m}}^{(q)}(x_0)\geq
m^{-(n-1)}\frac{|\alpha_{m, 1}(x_0)|^{2}}{\|\alpha_{m, 1}\|^{2}_X},
\end{equation}
and note that $\|\alpha_{m,1}\|_X\rightarrow 1$, we conclude that
\begin{equation}\label{e-gue171022vwII}
\liminf_{m\rightarrow\infty}m^{-(n-1)}\hat\Pi_{\beta,m,\leq m\nu_{m}}^{(q)}(x_0)
\geq\frac{1}{2\pi^{n}}|\det \mathcal{L}_{x_0}|.
\end{equation}
From \eqref{e-gue171022vwII} and \eqref{e-gue171021bIIq}, we get \eqref{p}.
\end{proof}

Let $\nu_m$ be a sequence with $\lim_{m\To+\infty}\nu_m=0$ and $\lim\limits_{m\rightarrow+\infty}\frac{\tilde \delta_m}{\nu_m}=0$, where $\tilde\delta_m$ is a sequence as in Theorem~\ref{o}.
By integrating Theorem~\ref{o} and using the Lebesgue Dominated Convergence Theorem, we obtain
\begin{equation}\label{e-gue171023}
{\rm dim\,}\hat H^q_{\beta, m, \leq m\nu_m}(X)=\frac{m^{n-1}}{2\pi^n}\int_{X(q)}|\det\mathcal L_x|dv_X(x)+o(m^{n-1}), ~\text{as}~m\rightarrow\infty.
\end{equation}
From \eqref{e-gue171023} and by
applying the algebraic argument in Lemma 3.2.12 in~\cite{MM07} and~\cite{Ma96}, we conclude that for every $q=0,1,2,\ldots,n-1$, we have
\begin{equation}\label{e-gue171023I}
\sum_{j=0}^q(-1)^{q-j}{\rm dim\,}{\rm Ker\,}\hat\Box^{(j)}_{\beta,m}\leq\frac{m^{n-1}}{2\pi^n}\sum_{j=0}^q(-1)^{q-j}\int_{X(j)}|\det\mathcal L_x|dv_X(x)+o(m^{n-1}).
\end{equation}
From \eqref{e-gue171023I} and Theorem~\ref{t-gue171014pm}, we get \eqref{e-gue170930III}. The proof of Theorem~\ref{t-gue170930} is complete.

We conclude this section with the following
\begin{proof}[Proof of Theorem \ref{t-gue2001}]
We recall the notations used in Section $4$. Consider the operator
\begin{equation*}
\Box_{-}^{(q)}:=(\ddbar\rho)^{\wedge,\star}\gamma\ddbar\tilde P: \Omega^{0,q}(X)\rightarrow
\Omega^{0,q}(X).
\end{equation*}
From Proposition $4.1$ in Part II of \cite{H08}, we see that 
\begin{equation}\label{e-201}
\Box_{-}^{(q)}=\frac{1}{2}(iT+\sqrt{-\Delta_X})+L_0,
\end{equation}
where $L_0$ is a classical pseudodifferential operator of order zero.
There is a constant $C_0>0$ such that
\begin{equation}\label{e-202}
\|L_0v\|_X\leq C_0\|v\|_X, \forall v\in \Omega^{0,q}(X).
\end{equation}
Let $m_1>2\max\{C_0, m_0\}$, where $m_0$ is as in Lemma \ref{t-gue171011I}.
Fix $m\in\mathbb{Z}, m<0, |m|\geq m_1$. Let $u\in\Ker\Box_m^{(q)}$. We have 
$\Box_{f}^{(q)}u=\Box_{m}^{(q)}u=0$.
Then by \eqref{e-gue171011II} we see that $u=\tilde P\gamma u$. Let $v=\gamma u$.
Then
\begin{equation}\label{e-203}
\Box_{-}^{(q)}v=(\ddbar\rho)^{\wedge,\star}\gamma\ddbar\tilde P\gamma u=(\ddbar\rho)^{\wedge,\star}\gamma\ddbar u=0.
\end{equation}
From \eqref{e-201}, \eqref{e-202} and \eqref{e-203}, we get
\begin{equation}\label{e-204}
\begin{split}
\frac{1}{2}(iT+\sqrt{-\Delta_X})v&=-L_0 v,\\
\abs{\frac{1}{2}(iTv|v)_X+\frac{1}{2}(\sqrt{-\Delta_X}v|v)_X}&\leq C_0\|v\|_X^2.
\end{split}
\end{equation}
Note that $iTv=-mv$, $(\sqrt{-\Delta_X}v|v)_X\geq 0$ and 
\[\abs{\frac{1}{2}(iTv|v)_X+\frac{1}{2}(\sqrt{-\Delta_X}v|v)_X}=\frac{1}{2}(iTv|v)_X+\frac{1}{2}(\sqrt{-\Delta_X}v|v)_X\geq\frac{|m|}{2}\|v\|^2_X.\]
From this observation and \eqref{e-204} we get
\begin{equation}\label{e-205}
\frac{|m|}{2}\|v\|^2_X\leq C_0\|v\|^2_X.
\end{equation}
Since $|m|>2C_0$, we have $\|v\|^2_X=0$ and hence $u=\tilde Pv=0$.
The theorem follows.
\end{proof}}

\section{Appendix}
In this section, we prove Theorem \ref{t-gue171008pm}.
Let $\ddbar\rho^{\wedge}: T^{*0,q}M'\To T^{*0,q+1}M'$ be the operator given by wedge multiplication by $\ddbar\rho$ and
let $\ddbar\rho^{\wedge,\star}:T^{*0,q+1}M'\To T^{*0,q}M'$ be its adjoint with respect to $\langle\,\cdot\,|\,\cdot\,\rangle$, that is,
\begin{equation}\label{e-gue171006e}
\langle\,\ddbar\rho\wedge u\,|\,v\,\rangle=\langle\,u\,|\,\ddbar\rho^{\wedge,\star} v\,\rangle,\ \ u\in T^{*0,q}M',\ \ v\in T^{*0,q+1}M'.
\end{equation}
Denote by $\gamma$ the operator of restriction to $X$. By using the calculation in page 13 of~\cite{FK72}, we can check that
\begin{equation}\label{Neumann condition}
\begin{split}
&{\rm Dom\,}\ddbar_m^{\star}\cap\Omega^{0,q+1}_{m}(\ol M)=\{u\in\Omega^{0,q+1}_{m}(\ol M);\, \gamma \ddbar\rho^{\wedge,\star}u=0 \},\\
&{\rm Dom\,}\Box^{(q)}_{m}\cap \Omega^{0,q}_{m}(\ol M)=\{u\in \Omega^{0,q}_{m}(\ol M);\,
\gamma\ddbar\rho^{\wedge,\star}u=0,  \gamma\ddbar\rho^{\wedge,\star}\ddbar u=0 \}.
\end{split}
\end{equation}
Let $\ddbar^\star_f:\Omega^{0,q+1}(M')\To\Omega^{0,q}(M')$ be the formal adjoint of $\ddbar$ with respect to $(\,\cdot\,|\,\cdot\,)_{M'}$, that is,
\[(\,\ddbar u\,|\,v\,)_{M'}=(\,u\,|\,\ddbar^\star_fv\,)_{M'},\ \ \forall u\in\Omega^{0,q}_0(M'),\ \ \forall v\in\Omega^{0,q+1}(M').\]
It is easy to see that if $u\in{\rm Dom\,}\ddbar_m^{\star}\cap\Omega^{0,q+1}_{m}(\ol M)$, then
\begin{equation}\label{e-gue171007s}
\ddbar^{\star}_mu=\ddbar^\star_fu.
\end{equation}

Fix $p\in X$. From (\ref{e-gue171006aI}) it is easy to see that  there exist an open neighborhood $U$ of $p$ in $M'$ and a local orthonormal frame $\{\omega^j\}_{j=1}^n$ of $T^{*1,0}M'$ on $U$ with $\omega^n=\frac{1}{\langle\partial\rho|\partial\rho\rangle^{\frac12}}\partial\rho$. Let $\{L_j\}_{j=1}^n\subset T^{1,0}M'$ be the dual frame of $\{\omega^j\}_{j=1}^n$. We need

\begin{lemma}\label{lem1-170912}
Let $p\in X$.  Let $U$ and $\{L_j\}_{j=1}^n$ be as above. Then there exist holomorphic coordinates $z=(z_1,\ldots,z_n)$ centered at $p$ defined in an open neighborhood $U_0\Subset U$ of $p$ in $M'$ such that
\begin{equation}\label{e-gue171006ycbI}
\begin{split}
&L_j=\frac{\partial}{\partial z_j}+O(\abs{z}),\ \ j=1,\ldots,n,\\
&i(L_n-\overline L_n)=cT+\sum_{j=1}^{n-1}(b_jL_j+\overline b_j\overline L_j)+O(\abs{z}),
\end{split}
\end{equation}
 where $c\neq0$ and $b_j$, $j=1,\ldots,n-1$, are constants.
\end{lemma}

\begin{proof}
First, we can choose holomorphic coordinates $z=(z_1,\ldots,z_n)$ centered at $p$ defined in an open neighborhood $U_0\Subset U$ of $p$ in $M'$ such that
\begin{equation}\label{c1-gue170912}
L_j(p)=\frac{\partial}{\partial z_j}|_p,\ \ j=1,2,\ldots,n.
\end{equation}
Note that $L_j(\rho)=0$, $j=1,\ldots,n-1$ and  $L_n(\rho)=\frac{1}{\sqrt{2}}+O(\rho)$. From this observation and \eqref{c1-gue170912}, we get
 $\frac{\partial\rho}{\partial z_j}(p)=0, \forall j=1, \cdots, n-1$ and $\frac{\partial\rho }{\partial z_n}(p)=\frac{1}{\sqrt 2}$ and hence
\begin{equation}\label{e-gue171007w}
\rho=\frac{\partial\rho}{\partial z_n}(p)z_n+\frac{\partial\rho}{\partial\overline z_n}(p)\overline z_n+O(|z|^2)=\sqrt{2}{\rm Re\,}z_n+O(|z|^2).
\end{equation}

Write $T=a_n\frac{\partial}{\partial z_n}+\overline a_n\frac{\partial}{\partial z_n}+\sum_{j=1}^{n-1}(a_j\frac{\partial}{\partial z_j}+\overline a_j\frac{\partial}{\partial \overline z_j})$, where $a_j\in C^\infty(U_0)$, $j=1,\ldots,n$.
Since $(T\rho)(p)=0$, we get $a_n(p)+\ol a_n(p)=0$ and hence $a_n(p)=\frac{i}{c}$, for some constant $c$. Since $T$ is transversal to $T^{1,0}X\oplus T^{0,1}X$ and $T^{1,0}_pX={\rm span\,}\set{L_1(p),\ldots,L_{n-1}(p)}$, we deduce that $c\neq0$.
Since $L_n=\frac{\partial}{\partial z_n}+O(|z|)$, the lemma follows.
\end{proof}

\begin{lemma}\label{main theorem1-170913}
There exists $C_{m}>0$ such that for every $u\in{\rm Dom\,}\Box^{(q)}_m\cap\Omega^{0, q}_m(\overline M)$, we have
\begin{equation}\label{e-gue171007a}
\|u\|_{1,\ol M}\leq C_{m}(\|\Box^{(q)}_mu\|_{M}+\|u\|_{M}).
\end{equation}
\end{lemma}

\begin{proof}
Step 1. We first prove \eqref{e-gue171007a} when $q=0$. Let $u\in{\rm Dom\,}\Box^{(0)}_m\cap C^\infty_m(\overline M)$.
Fix $p\in X$. Let $z=(z_1,\ldots,z_n)$ be holomorphic coordinates centered at $p$ defined on an open neighborhood $U_0$ of $p$ in $M'$ such that \eqref{e-gue171006ycbI} holds. We will use the same notations as in Lemma~\ref{lem1-170912}. Let $\chi\in C^\infty_0(U_0)$ and put $v:=\chi u$. It is easy to see that
\begin{equation}\label{e-gue171008}
\begin{split}
\|v\|^2_{1,\ol M}&\leq c_1\Bigr( \| v\|^{2}_M+\sum_{j=1}^{n}\|\bar L_{j} v\|^{2}_M+\sum_{j=1}^{n}\|L_{j}v\|^{2}_M\Bigr)\\
&\leq c_2\Bigr( \|v\|^{2}_M+\|\ddbar v\|^{2}_M+\sum_{j=1}^{n-1}\|L_{j} v\|^{2}_M+\|L_{n}v\|^{2}_M\Bigr),
\end{split}
\end{equation}
where $c_1>0$, $c_2>0$ are constants. Let $L\in C^\infty(U_0,\Complex TM')$, we have the following formula due to Stokes' theorem
\begin{equation}\label{e-gue171008I}
\int_{M}(Lf)gdv_{M'}=-\int_{M}f(Lg)dv_{M'}+\int_{M'}hfgdv_{M'}+\int_{X}L(\rho)fgdv_{X},\  \ f, g\in C^\infty_0(U_0),
\end{equation}
where $h$ is a smooth function. From \eqref{e-gue171008I} and notice that $L_j(\rho)=0$, $j=1,\ldots,n-1$, we can check that when $j=1,\ldots,n-1$,
\begin{equation}\label{e-gue171008II}
\begin{split}
(\,L_{j}v\,|\,L_{j}v\,)_M&=-(\,v, |\,\ol L_{j}L_{j}v\,)+O(\|v\|_M\cdot\|L_{j}v\|_M)\\
&=-(\,v\,|\,[\bar L_{j},L_{j}]v\,)_M-(\,v\,|\, L_{j}\ol L_{j}v\,)_M+O(\|v\|_M\cdot\|L_{j}v\|_M)  \\
&=(\,v\,|\,[L_{j},\bar L_{j}]v\,)_M+\|\ol L_{j}v\|^{2}_M+O(\|v\|_M\|\ol L_{j}v\|_M)+O(\|v\|_M\cdot\|L_{j}v\|_M)\\
&=\|\overline L_jv\|^2_M+O(\|v\|_M\cdot\|v\|_{1,\ol M}).
\end{split}
\end{equation}
From \eqref{e-gue171008II} and \eqref{e-gue171008}, we deduce that
\begin{equation}\label{e-gue171008III}
\|v\|^{2}_{1,\ol M}\leq c_3\Bigr(\|L_{n}v\|^{2}_M+\|v\|^{2}_M+\|\ddbar v\|^{2}_M\Bigr),
\end{equation}
where $c_3>0$ is a constant. From \eqref{e-gue171006ycbI}, we see that
\begin{equation}\label{e-gue171008a}
\begin{split}
\norm{L_nv}^2_M&\leq c_4\Bigr(\norm{\ol L_nv}^2_M+\norm{(L_n-\ol L_n)v}^2_M\Bigr)\\
&\leq c_5\Bigr(\norm{\ol L_nv}^2_M+\norm{Tv}^2_M+\sum^{n-1}_{j=1}(\norm{L_jv}^2_M+\norm{\ol L_jv}^2_M)+\varepsilon(U_0)\norm{v}^2_{1,\ol M}\Bigr),
\end{split}
\end{equation}
and
\begin{equation}\label{e-gue171030}
\begin{split}
\norm{\ol L_nv}^2_M&\leq\hat c_4\Bigr(\norm{L_nv}^2_M+\norm{(L_n-\ol L_n)v}^2_M\Bigr)\\
&\leq\hat c_5\Bigr(\norm{L_nv}^2_M+\norm{Tv}^2_M+\sum^{n-1}_{j=1}(\norm{L_jv}^2_M+\norm{\ol L_jv}^2_M)+\hat\varepsilon(U_0)\norm{v}^2_{1,\ol M}\Bigr),
\end{split}
\end{equation}
where $c_4>0$, $c_5>0$, $\hat c_4>0$, $\hat c_5>0$ are constants and $\varepsilon(U_0)>0$, $\hat\varepsilon(U_0)>0$ are constants depending on $U_0$ with $\varepsilon(U_0)\To0$ if $U_0\To\set{p}$, $\hat\varepsilon(U_0)\To0$ if $U_0\To\set{p}$. From \eqref{e-gue171008II},
\eqref{e-gue171008III} and \eqref{e-gue171008a}, we see that if $U_0$ is small enough, then
\begin{equation}\label{e-gue171008aI}
\|v\|^{2}_{1,\ol M}\leq c_6\Bigr(\|Tv\|^{2}_M+\|v\|^{2}_M+\|\ddbar v\|^{2}_M\Bigr),
\end{equation}
where $c_6>0$ is a constant. From \eqref{e-gue171008aI}, by using partition of unity and notice that $Tu=imu$, we deduce that
\begin{equation}\label{e-gue171008aII}
\|u\|^{2}_{1,\ol M}\leq C_{m}\Bigr(\|u\|^{2}_M+\|\ddbar_mu\|^{2}_M\Bigr),
\end{equation}
where $C_m>0$ is a constant.
Since $u\in{\rm Dom\,}\Box^{(0)}_m$, we have $\|\ddbar_mu\|^{2}_M=(\,\Box^{(0)}_mu\,|\,u\,)_M$. From this observation and \eqref{e-gue171008aII}, we get
\eqref{e-gue171007a} for $q=0$.

Step 2. Now we prove \eqref{e-gue171007a} for $q>0$. Let $u\in{\rm Dom\,}\Box^{(q)}_m\cap\Omega^{0,q}_m(\overline M)$.
Fix $p\in X$. Let $z=(z_1,\ldots,z_n)$ be holomorphic coordinates centered at $p$ defined on an open neighborhood $U_0$ of $p$ in $M'$ such that \eqref{e-gue171006ycbI} holds. We will use the same notations as in Lemma~\ref{lem1-170912}. Let $\chi\in C^\infty_0(U_0)$ and put $v:=\chi u$. On $U_0$, we write $u=\sideset{}{'}\sum_{|J|=q} u_J\overline \omega^J$, where $\sideset{}{'}\sum$ means that the summation is performed only over strictly increasing multiindices and for $J=(j_1,\ldots,j_q)$, $\ol\omega^J=\ol\omega^{j_1}\wedge\cdots\wedge\ol\omega^{j_q}$. For every strictly increasing multiindex $J$, $\abs{J}=q$, put $v_J:=\chi u_J$. Then $v=\sideset{}{'}\sum_{|J|=q} v_J\overline \omega^J$. 
We have
\begin{equation}\label{e-gue171008yb}
\begin{split}
\norm{v}^2_{1,\ol M}\leq C_1\Bigr(\sideset{}{'}\sum_{|J|=q, j\in\set{1,\ldots,n}}\|\overline L_jv_J\|^2_M+\sideset{}{'}\sum_{|J|=q, j\in\set{1,\ldots,n}}\|L_jv_J\|^2_M+\norm{v}^2_M\Bigr),
\end{split}
\end{equation}
where $C_1>0$ is a constant. Moreover, it is easy to see that
\begin{equation}\label{e1-170913}
\begin{split}
&\|\overline\partial v\|^2_M+\|\overline\partial^{\star}_fv\|^2_M\\
&=\sideset{}{'}\sum_{|J|=q, j\notin J,j\in\set{1,\ldots,n}}\|\overline L_jv_J\|^2_M+
\sideset{}{'}\sum_{|J|=q, j\in J,j\in\set{1,\ldots,n}}\|L_jv_J\|^2_M+O(\|v\|_M\cdot\|v\|_{1,\ol M}).
\end{split}
\end{equation}
From \eqref{e-gue171008II}, we see that for $j=1,\ldots,n-1$ and every strictly increasing multiindex $J$, $\abs{J}=q$, we have
\begin{equation}\label{e-gue171008ybI}
\begin{split}
&\norm{L_jv_J}^2_M=\norm{\ol L_jv_J}^2_M+O(\norm{v}_M\norm{v}_{1,\ol M}),\\
&\norm{\ol L_jv_J}^2_M=\norm{L_jv_J}^2_M+O(\norm{v}_M\norm{v}_{1,\ol M}).
\end{split}
\end{equation}
From \eqref{e-gue171008ybI} and \eqref{e1-170913}, we deduce that
\begin{equation}\label{e1-170913r}
\begin{split}
&\|\overline\partial v\|^2_M+\|\overline\partial^{\star}_fv\|^2_M\\
&=\frac{1}{2}\sideset{}{'}\sum_{|J|=q, j\in\set{1,\ldots,n-1}}\Bigr(\|\overline L_jv_J\|^2_M+\|L_jv_J\|^2_M\Bigr)\\
&+\sideset{}{'}\sum_{|J|=q, n\notin J}\|\ol L_nv_J\|^2_M+\sideset{}{'}\sum_{|J|=q, n\in J}\|L_nv_J\|^2_M+O(\|v\|_M\cdot\|v\|_{1,\ol M}).
\end{split}
\end{equation}
From \eqref{e-gue171008yb}, \eqref{e1-170913r}, \eqref{e-gue171008a} and \eqref{e-gue171030}, we see that if $U_0$ is small enough, then
\begin{equation}\label{e-gue171008lv}
\|v\|^{2}_{1,\ol M}\leq C_2\Bigr(\|Tv\|^{2}_M+\|v\|^{2}_M+\|\ddbar v\|^{2}_M+\norm{\ddbar^\star_fv}^2_M\Bigr),
\end{equation}
where $C_2>0$ is a constant. Since $Tu=imu$, we conclude that $\|Tv\|^{2}_M\leq C_3\norm{u}_M$, where $C_3>0$ is a constant. From this observation, \eqref{e-gue171008lv} and by using partition of unity, we conclude that
\begin{equation}\label{e-gue171008aIIq}
\|u\|^{2}_{1,\ol M}\leq C_{m}\Bigr(\|u\|^{2}_M+\|\ddbar_mu\|^{2}_M+\norm{\ddbar^\star_mu}^2_M\Bigr),
\end{equation}
where $C_m>0$ is a constant.
Since $u\in{\rm Dom\,}\Box^{(q)}_m$, we have $\|\ddbar_mu\|^{2}_M+\norm{\ddbar^\star_mu}^2_M=(\,\Box^{(q)}_mu\,|\,u\,)_M$. From this observation and \eqref{e-gue171008aIIq}, we get
\eqref{e-gue171007a} for $q>0$.
\end{proof}

We can repeat the method of Folland-Kohn (see Section 5 in~\cite{FK72}) and deduce

\begin{lemma}\label{e-gue171008u}
Let $u\in{\rm Dom\,}\Box^{(q)}_m$ and let $\Box^{(q)}_mu=v\in L^2_{(0,q),m}(M)$. Then $u\in W^1(\ol M)$ and we have
\begin{equation}\label{e-gue171008f}
\|u\|_{1,\ol M}\leq c_{m}(\|\Box^{(q)}_mu\|_M+\|u\|_M),
\end{equation}
where $c_m>0$ is a constant independent of $u$. Moreover, if $v\in\Omega^{0,q}_m(\ol M)$, then $u\in\Omega^{0,q}_m(\ol M)$.
\end{lemma}

By standard arguments in functional analysis, we get the following

\begin{proposition}\label{f170918}
The operator $\Box^{(q)}_{m}: {\rm Dom\,}\Box^{(q)}_m\To L^2_{(0,q),m}(M)$ has $L^{2}$ closed range.
Moreover,
${\rm Ker\,}\Box^{(q)}_{m}\subset \Omega^{0.q}_{m}(\ol M)$ and $\dim{\rm Ker\,}\Box^{(q)}_{m}<\infty$.

Furthermore, there exists a bounded operator $N^{(q)}_{m}: L^2_{(0, q), m}(M)\rightarrow{\rm Dom\,}\Box^{(q)}_m$  with $N^{(q)}_{m}(\Omega^{0,q}_{m}(\ol M))\subset\Omega^{0,q}_{m}(\ol M)$ such that
\begin{equation}\label{e-gue171008pmy}\left\{
\begin{aligned}
&\Box^{(q)}_{m}N^{(q)}_{m}+B^{(q)}_{m}=I\ \ \mbox{on $L^{2}_{(0,q),m}(M)$}, \\
&N^{(q)}_{m}\Box^{(q)}_{m}+B^{(q)}_{m}=I \ \ \mbox{on ${\rm Dom\,}\Box^{(q)}_{m}$}, \\
\end{aligned}
\right.
\end{equation}
where $B^{(q)}_{m}:L^{2}_{(0,q),m}(M)\rightarrow{\rm Ker\,}\Box^{(q)}_{m}$ is the orthogonal projection with respect to $(\,\cdot\,|\,\cdot\,)_M$.
\end{proposition}

From Proposition~\ref{f170918}, \eqref{e-gue171006ycb}, \eqref{Neumann condition} and \eqref{e-gue171007s}, we deduce that
\begin{equation}\label{Neum1}
{\rm Ker\,}\Box^{(q)}_{m}= \{u\in\Omega^{0,q}_{m}(\ol M);\,
\ddbar u=0, \ddbar^{\star}_fu=0, \gamma\ddbar\rho^{\wedge,\star}u=0\}.
\end{equation}
\begin{theorem}\label{t-gue171008pm-2020}
	We have
	\begin{equation}\label{iso170918}
	{\rm Ker\,}\Box^{(q)}_{m}\cong H^{q}_{m}(\ol M).
	\end{equation}
	and  ${\rm dim\,}H^q_m(\ol M)<+\infty$.
\end{theorem}

\begin{proof}
	Consider the linear map:
	\[\begin{split}
	T^{(q)}_m: {\rm Ker\,}\ddbar_m:=\set{u\in\Omega^{0,q}_m(\ol M);\, \ddbar_mu=0}&\To{\rm Ker\,}\Box^{(q)}_m,\\
	u&\To B^{(q)}_mu,
	\end{split}\]
	where $B^{(q)}_m$ is as in \eqref{e-gue171008pmy}. It is clear that the map is well-defined and surjective. We claim that
	\begin{equation}\label{e-gue171008pmyI}
	{\rm Ker\,}T^{(q)}_m={\rm Im\,}\ddbar_m:=\set{\ddbar_mu\in\Omega^{0,q}_m(\ol M);\, u\in\Omega^{0,q-1}_m(\ol M)}.
	\end{equation}
	Let $u\in{\rm Ker\,}T^{(q)}_m$. Then $B^{(q)}_mu=0$. From \eqref{e-gue171008pmy}, we see that
	\begin{equation}\label{e-gue171008pmyII}
	u=\Box^{(q)}_mN^{(q)}_mu=(\ddbar_m\,\ddbar^\star_m+\ddbar^\star_m\,\ddbar_m)N^{(q)}_mu.
	\end{equation}
	From $\ddbar_mu=0$ and \eqref{e-gue171008pmyII}, we get
	\begin{equation}\label{e-gue171008pmyIII}
	\ddbar_m\,\ddbar^\star_m\,\ddbar_mN^{(q)}_mu=0.
	\end{equation}
	From \eqref{e-gue171008pmyIII} and note that $\ddbar_mN^{(q)}_mu\in{\rm Dom\,}\ddbar^\star_m$, we have
	\begin{equation}\label{e-gue171008pmya}
	(\,\ddbar^\star_m\,\ddbar_mN^{(q)}_mu\,|\,\ddbar^\star_m\,\ddbar_mN^{(q)}_mu\,)_M=(\,\ddbar_m\,\ddbar^\star_m\,\ddbar_mN^{(q)}_mu\,|\,\ddbar_mN^{(q)}_mu\,)_M=0.
	\end{equation}
	From \eqref{e-gue171008pmya} and \eqref{e-gue171008pmyII}, we get
	\begin{equation}\label{e-gue171008z}
	u=\ddbar_m\,\ddbar^\star_mN^{(q)}_mu.
	\end{equation}
	Note that $\ddbar^\star_mN^{(q)}_mu\in\Omega^{0,q-1}_m(\ol M)$. From this observation and \eqref{e-gue171008z}, the claim \eqref{e-gue171008pmyI} follows and we get the theorem.
\end{proof}

\begin{center}
{\bf Acknowledgement}
\end{center}

This work was started when the authors visited Institute of Mathematics of National University of Singapore for the program ``Complex Geometry, Dynamical System and Foliation Theory". Parts of this work were finished when the second and third author visited Institute of Mathematics, Academia Sinica  from July to September 2017. The authors would like to express their gratitude to both of the institutes for hospitality, a comfortable accommodation and financial support during their visit. 
We would like to express our deep gratitude to the Referee for the thorough reading
and useful comments and suggestions, which were very helpful to us.

\bibliographystyle{amsalpha}

\end{document}